\documentclass[12pt]{amsart}
\usepackage{amsmath, amssymb, amsthm}
\usepackage{enumerate}
\usepackage{bbm}
\usepackage{fancyhdr}

\theoremstyle{plain}
\newtheorem{theorem}{Theorem}[section]
\newtheorem{prop}{Proposition}[section]
\newtheorem{lem}{Lemma}[section]

\theoremstyle{definition}
\newtheorem{mydef}{Definition}[section]

\numberwithin{equation}{section}
\newcommand{\indicator}[1]{\mathbbm{1}_{ {#1} }}
\DeclareMathOperator*{\R}{\mathbb{R}}

\title[Weighted Local Estimates for Singular Integrals]{Weighted Local Estimates for Singular Integral Operators}
\author[Poelhuis and Torchinsky]{Jonathan Poelhuis and Alberto Torchinsky}
\begin{document}

\maketitle

\setcounter{section}{1}
\section*{Introduction}

An underlying principle of the Calder\'{o}n-Zygmund theory, first expressed  by Cotlar for the Hilbert transform \cite{MC},
is that the Hardy-Littlewood maximal function controls the Calder\'{o}n-Zygmund singular integral operators.
 Coifman    formulated this principle in the weighted setting as follows. We say that
 a  continuous function $\Phi$ satisfies condition
 $C$ if it is   increasing on $[0, \infty)$ with  $\Phi(0) = 0$ and $\Phi(2t) \le c\,\Phi(t)$, all $t>0$.
 Then, if $\Phi$ satisfies condition $C$,  $w$ is an $A_\infty$   weight,
 and $T$ is a Calder\'on-Zygmund singular integral operator,
\begin{equation}  \int_{\R^n} \Phi     ( |Tf(x)|  )\,w(x)\,dx\le c \int_{\R^n} \Phi    ( Mf(x)     )\,w(x)\,dx\,,
\end{equation}
where $Mf$ denotes the Hardy-Littlewood maximal function of $f$ \cite{RRC}.
In the setting of the $A_p$   weights, extending the result of  Hunt, Muckenhoupt and Wheeden
 for the Hilbert transform \cite{HMW},
Coifman  and Fefferman  proved that
\begin{equation}\int_{\R^n} |Tf(x) |^p\, w(x)\,dx\le c \int_{\R^n} |f(x)|^p \, w(x)\,dx
\end{equation}
 provided that $1<p<\infty$ and $w\in A_p$,  \cite{CoifmanFefferman}.
The proof of both of these results makes use of the good-$\lambda$ inequalities of Burkholder and Gundy \cite{GoodLambda}.

Along similar lines, and with the additional purpose of considering vector-valued singular integrals,
C\'{o}rdoba and Fefferman
proved that for a weight $w$,  i.e., a nonnegative locally integrable function $w$, and $1<p<\infty$,
\begin{equation} \int_{\R^n} |Tf(x) |^p\, w(x)\,dx\le c \int_{\R^n} |f(x)|^p M_r w(x)\, dx\,,
\end{equation}
 where  $M_r w(x)=M(w^r)(x)^{1/r}$ denotes the Hardy-Littlewood maximal function of order $r$, $1<r<\infty$,  \cite{CordobaFefferman}.
Their proof is based on the pointwise inequality
\begin{equation}  M^\sharp(Tf)(x)\le c \,M_r f(x)\,,
\end{equation}
 where $M^{\sharp}$ denotes the sharp maximal function. Since this inequality does not hold for $r=1$ and $T$ the Hilbert transform,
it is of interest that for Calder\'on-Zygmund singular integral operators the pointwise control
can be expressed
in terms of the local sharp maximal function $ M^\sharp_{0,s}$ by means of an  estimate
proved by Jawerth and Torchinsky \cite{JawerthTorchinsky} that preserves the weak-type information, to wit,
\begin{equation}
 M^\sharp_{0,s}(Tf)(x)\le c\, Mf(x)\,.
\end{equation}

In the first part of this paper we revisit this  inequality, recast it in local terms, and
establish local weighted  estimates for Calder\'on-Zygmund singular integral operators.
In particular, our results include that  these operators   satisfy
\begin{equation*}
M^{\sharp}_{0,s,Q_0}  (Tf )(x) \le c \sup_{x \in Q, Q \subset Q_0} \inf_{y \in Q} M f(y)\,,
\end{equation*}
where $M^{\sharp}_{0,s,Q_0}$  denotes the local sharp maximal function restricted to the cube $Q_0$.

Combined with the inequality
\begin{equation*}
\int_{Q_0}\Phi  (|f(x) - m_f(t,Q_0)|  )\, w(x)\,dx \le  c \int_{Q_0}\Phi ( M^{\sharp}_{0,s,Q_0}f(x)  )\,v(x)\,dx\,,
\end{equation*}
where $\Phi$ is any function satisfying condition $C$,
 $ m_{f}(t,Q_0)$ denotes the (maximal) median of $ f$ with parameter $t$, and
 $v=w$ when $w\in A_\infty$ and $v=M_r w$ when $w$ is an arbitrary weight,
 it readily follows that
\begin{equation}
\int_{Q_0} \Phi  ( |Tf(x) - m_{Tf}(t,Q_0) | ) \,w(x)\,dx \le c \int_{Q_0}   \Phi  ( M  f(x)  )\, v(x)\, dx\,,
\end{equation}
  where $c$ is independent of the cube $Q_0$ and $f$.

 Furthermore, if
$\lim_{Q_0\to\R^n} m_{Tf}(t,Q_0)=0$,  then
\begin{equation}
\int_{\R^n} \Phi   ( |Tf(x) | ) \,w(x)\,dx \le c \int_{\R^n}   \Phi ( M  f(x) ) \, v(x)\, dx\,.
\end{equation}
We also prove
 that (1.6) and (1.7) hold  for appropriate non-$A_\infty$ weights $w$ and $ v$.

Estimates such as (1.6) represent a local version
 of  (1.1)  and (1.3) and  fail for
singular integrals  when $Mf$ is replaced by $|f|$ on the
right-hand side, even with $\lambda Q_0$ in place of  $Q_0$ for any $\lambda >1$ there.

As for the integral inequality (1.7),
it includes all three estimates, (1.1), (1.2), and (1.3), and readily implies
 the following one.   A    function $\Phi$ that satisfies condition $C$ which is convex and such that
 $\Phi(t)\to \infty$ as $t \to \infty$, or, more generally,
   such that $\Phi(t)/t\to\infty$ as $t \to\infty$, is called a Young function.
 Let $w\in A_\infty$ and $\Phi,\Psi$ be Young functions such that
\[\int_0^t \frac{\Phi(s)}{s^2}\,ds\le c\,\frac{\Psi(t)}{t}\,,\quad t>0\,,\]
and $\Phi(t)/t^q$ decreases for some $q>1$. Then,
\begin{equation}
\int_{\R^n} \Phi   ( |Tf(x) |   ) \,w(x)\,dx \le c \int_{\R^n}   \Psi   ( | f(x) |  ) \, Mw(x)\, dx\,.
\end{equation}

(1.6) also allows   expressing the limiting cases of (1.4). Indeed, we have
\[  M^\sharp   (|Tf|^p   )(x)^{1/p}\le M^\sharp_p     (Tf   )(x)\le c_p\, Mf(x)\,,\quad 0<p<1\,,\]
with
\[  M^\sharp_p   (g )(x) =\sup_{x\in Q} \inf_c \Big(\frac1{|Q|}\int_Q |g(y)-c|^p\,dy\Big)^{1/p}\,,
\]
 and  $c_p\to\infty$ as $p\to 1$, and
\begin{equation} M^\sharp   (Tf   )(x)\le c\, M_{L\log L}f(x)\,.
\end{equation}

Now, by (1.9), if  $\Phi(t)/t^p$ increases and $\Phi(t)/t^q$ decreases for some $1<p<q<\infty$,
and $w$ is an arbitrary weight,
by  Theorem 1.7 in \cite{Perez1990} it readily follows that
\begin{equation}
\int_{\R^n} \Phi   ( M^\sharp   (Tf   )(x)  )\,w(x)\,dx\le c\,\int_{\R^n} \Phi   (| f(x)|   )\,Mw(x)\,dx\,.
\end{equation}
Note that   the weaker inequality with $M^{\sharp}_{0,s}$ in place of $M^{\sharp}$ on the left-hand side
above holds, by  (1.5) and the Fefferman-Stein maximal inequality, for   an even wider class of $\Phi$.
Now, these inequalities are  of interest because
they  do not hold for an arbitrary weight $w$ for all singular integral operators $T$ with
$ |Tf(x)|$ in place of  $ M^\sharp(Tf)(x)$ on the left-hand side of (1.10).
This observation
follows from  Theorem 1.1 in \cite{Perez1994}:
if  $T$  is a singular integral operator and  $1 < p < \infty$,
there exists a constant $c$ such that for each weight $w$,
\begin{equation}
 \int_{\R^n} |Tf(x)|^p\, w(x)\,dx \le   c
\int_{\R^n} |f(x)|^p\, M^{[p\,]+1}\,w(x)\,dx\,.
\end{equation}
Furthermore, the result is sharp since it does not
hold for $M^{[p\,]}$ in place of $ M^{[p\,]+1}$ in (1.11).

And, as illustrated below in the case
of singular integral operators of Dini type, or with kernels satisfying H\"ormander-type conditions,
and integral operators with homogeneous kernels,
our methods apply to other operators as well.

We then take a closer look at the integral inequalities (1.2) and (1.11) which, being
specific to   $p$,  are of a different nature. The
question of determining weights $(w,v)$ so that Calder\'{o}n-Zygmund singular integral operators map
$L^p_v$ continuously into $L^p_w$  was pioneered by Muckenhoupt and Wheeden \cite{MuckWhee}
for the Hilbert transform and continues to attract considerable attention. For weights that satisfy
some additional property, such as being  radial,  interesting results are
proved and referenced, for instance,  in \cite{GuliNaz2006}. Interestingly,
Reguera and Scurry  have shown that there is no
a priori relationship between the Hilbert transform  and the Hardy-Littlewood  maximal
function in the two-weight setting \cite{ReSc}.

To deal with singular integral operators
in the two-weight context we build on the ideas developed so far: we consider a local median decomposition and the
  weights naturally  associated to them.
The first scenario corresponds to
 the $W_p$ classes of Fujii \cite{Fujii1991}, where estimates with the flavor of extrapolation are
 obtained.
And, a further refined local median decomposition, like the one used by Lerner
in his  proof of the $A_2$ conjucture \cite{La}, coupled with  weights satisfying  the Orlicz
``bump'' condition introduced by P\'erez \cite{Perez1994} and used by Lerner \cite{Lerner2010},
gives the estimate for Calder\'on-Zygmund singular integrals, including
those of Dini type, or with kernels satisfying a H\"ormander-type condition,
from $L^p_v(\R^n)$ into $L^q_w(\R^n)$ for $1<p\le q<\infty$.

Finally, the local estimates   are well suited to
the generalized weighted Orlicz-Morrey spaces $\mathcal M^{\Phi, \phi}_w$
and generalized weighted Orlicz-Campana-to spaces  $\mathcal L^{\Phi, \phi}_w$, defined in Section 7.
Indeed, if  $T$ is a Calder\'on-Zygmund singular integral operator, from (1.6) it readily follows that for
a Young function $\Phi$, $w\in A_\infty$, and every appropriate $\phi$,
\begin{equation}
 \|Tf\|_{\mathcal L^{\Phi, \phi}_w}\le c\,\|M f\|_{\mathcal M^{\Phi, \phi}_w}\,.
 \end{equation}

Moreover, suppose that  $S$ is  a  sublinear operator that satisfies
\begin{equation} \int_{\R^n} \Phi   (|Sf(y)|  )\,w(y)\,dy\le c\int_{\R^n} \Phi (|f(y)| )\,w(y)\,dy
\end{equation}
 and  such that for any cube $Q$, if $x\in Q$ and  ${\text{supp}}(f)\subset \R^n\setminus 2Q$,
then
\begin{equation}
\ |Sf(x)|\le \int_{\R^n} \frac{|f(y)|}{|x-y|^n}\,dy\,.
\end{equation}
Then, if  $u_\Phi$ denotes the upper index  of
 $L^{\Phi}$,  $w\in A_p$ where $p=1/u_\Phi$, and for all $x\in\R^n$ and $l>0$, $\phi(x,t)$ and $\psi(x,l)$ satisfy
\[  {\psi(x,l)}\, \int_{l}^\infty \frac1{\phi(x,t)}\, \frac{dt}{t}\le c\,,
\]
we have that
\[ \|Sf\|_{\mathcal M^{\Phi, \psi}_w }\le c\,\|f\|_{\mathcal M^{\Phi, \phi}_w}\,.
\]
This result applies to the Hardy-Littlewood maximal function, Calde-r\'on-Zygmund singular integrals,
and  other operators \cite{GuliNaz2006}.

The paper is organized as follows. The first two sections contain the
essential ingredients in what follows. The
local median decomposition of an arbitrary measurable function
is given in Section 2 and the local control of a weighted mean of a function by the weighted mean of a local maximal function  is
done in Section 3. In Section 4 we prove a local version of the estimate
 $M^{\sharp}_{0,s}(Tf)(x) \le  c\,Mf(x)$ for Calder\'{o}n-Zygmund singular integral operators
 and recast similar estimates with $T$ replaced by singular integral operators with kernels satisfying H\"ormander-type conditions
or integral operators with homogeneous kernels and $M$ by an appropriate maximal function $M_T$.
In Section 5 we  use these estimates   to express the local integral control of
$Tf$ in terms of $M_Tf$;  the case  $M_T=M_r$ is done in some detail. In Section 6 we
use variants of the local median decomposition obtained in Section 2 to prove two-weight, $L^p_v$-$L^q_w$ estimates
for singular integral operators for $1<p\le q<\infty$. And finally, in Section 7 we consider the Orlicz-Morrey and Orlicz-Campanato spaces.

Some closely related topics are not addressed here. Because we
concentrate on integral inequalities,   weak-type inequalities are not considered.
Neither are commutator estimates, which can, for instance, be treated as in \cite{AlvarezPerez}, nor homogeneous spaces,
the foundation for which has been laid in \cite{ShiTor, TorchinskyStromberg}  and \cite{Anderson}.
And, for the various definitions or properties that the reader may find  unfamiliar,
there are many treatises in the area which may be helpful, including \cite{TorchinskyStromberg, TOR}.

\section{Local Median Decomposition}
The  decomposition of a measurable function
   presented here was first considered in terms of averages by Garnett and Jones \cite{GJ}
and suggested in terms of medians by Fujii \cite{Fujii1991}. It
complements Lerner's ``local mean oscillation'' decomposition \cite{Lerner2010, LernerSummary},
which corresponds to the case $t=1/2$, $s=1/4$ in Theorem 2.1.
Although the bound below is larger than his,
it holds for  arbitrarily small values of $s$,
which are necessary for the applications of interest to us. Also, the proof relies on medians and is somewhat more geometric.

In what follows, we adopt the notations of \cite{MedContOsc} and \cite{Stromberg}. In particular, all cubes have sides parallel to the axes. Also, for a cube $Q \subset \mathbb{R}^n$ and $0 < t < 1$, we say that
\[ m_f(t,Q) = \sup\{M : |\{y \in Q: f(y) < M\}| \le t|Q|\}\]
 is the (maximal) median of $f$ over $Q$ with parameter $t$. For a cube $Q_0 \subset \mathbb{R}^n$ and $0 < s \le 1/2$,
 the local sharp maximal function restricted to $Q_0$ of a measurable function $f$ at $x \in Q_0$ is
\begin{align*}
M^{\sharp}_{0,s,Q_0}&f(x)\\
& = \sup_{ {x \in Q, Q \subset Q_0}}\inf_c \; \inf   \{\alpha \ge 0: |\{y \in Q: |f(y) - c| > \alpha\}| < s|Q|   \}\,,
\end{align*}
and the local sharp maximal function of a measurable function $f$ at $x \in \R^n$ is
\[M^{\sharp}_{0,s}f(x) = \sup_{ {x \in Q}}\inf_c \; \inf   \{\alpha \ge 0: |\{y \in Q: |f(y) - c| > \alpha\}| < s|Q|  \}\,.\]

Additionally,  we consider the maximal function $m^{t,\Delta}_{Q_0}$ defined as follows.
Let $\mathcal{D}$ be the family of dyadic cubes in $\mathbb{R}^n$. For a cube $Q \subset \mathbb{R}^n$, let $\mathcal{D}(Q)$ denote the family of dyadic subcubes relative to $Q$; that is to say, those formed by repeated dyadic subdivision of
 $Q$ into $2^n$ congruent subcubes.  Then
\[m^{t,\Delta}_{Q_0} f(x) = \sup_{{x \in Q, Q \in \mathcal{D}(Q_0)}} |m_f(t,Q)|\,.\]

A related non-dyadic maximal function was introduced by A. P. Calder\'{o}n  in order to exploit cancellation
to obtain estimates for singular integrals
in terms of maximal functions \cite{Calderon1972}.

Finally,   $\widehat{Q}$ denotes the dyadic parent of a cube $Q$.

\begin{theorem}
Let $f$ be a measurable function on a fixed cube $Q_0 \subset \mathbb{R}^n$, $0 < s < 1/2$, and $1/2 \leq t < 1-s$.
Then there exists a (possibly empty) collection of subcubes $\{Q^v_j\} \subset \mathcal{D}(Q_0)$ and a family of collections
of indices $\{I^v_2\}_v$ such that
\begin{enumerate}
\item[\rm (i)]
for a.e. $\! x \in Q_0$,
\[|f(x) - m_f(t,Q_0)| \leq 4M^{\sharp}_{0,s,Q_0}f(x) + \sum_{v=1}^{\infty}\sum_{j \in I^v_2} a^v_j\indicator{Q^v_j}(x) \,,\]
where
\begin{align}
a^v_j &\le 10n\inf_{y \in Q^v_j}M^{\sharp}_{0,s,\widehat{Q^v_j}}f(y) + 2 \inf_{y \in Q^v_j}M^{\sharp}_{0,s,Q^v_j}f(y)\notag
\\
&\le (10n+2) \inf_{y \in Q^v_j}M^{\sharp}_{0,s,\widehat{Q^v_j}}f(y) \,;
\end{align}
\item[\rm (ii)]
for fixed $v$, the $\{Q^v_j\}_j$ are nonoverlapping;
\item[\rm(iii)]
if\, $\Omega^v = \bigcup_j Q^v_j$, then $\Omega^{v+1} \subset \Omega^v$; and,
\item[\rm (iv)]
for all $j$, $|\Omega^{v+1} \cap Q^v_j| \le  ({s}/{(1-t)}) \, |Q^v_j|$ .
\end{enumerate}
\end{theorem}

\begin{proof}
Let $E^1 = \{x \in Q_0: |f(x) - m_f(t,Q_0)| > 2\inf_{y \in Q_0}M^{\sharp}_{0,s,Q_0}f(y)\}$.
If $|E^1|=0$, the decomposition halts -- trivially, for a.e. $\! x \in Q_0$,
\begin{equation*}
|f(x) - m_f(t,Q_0)| \le 2 \inf_{y \in Q_0}M^{\sharp}_{0,s,Q_0}f(y)\,.
\end{equation*}

So suppose that $|E^1| > 0$. Recall that by Lemma 4.1 in \cite{MedContOsc}, for $\eta > 0$,
\[ |\{x \in Q_0: |f(x) - m_f(t,Q_0)| \ge 2 \inf_{y \in Q_0}M^{\sharp}_{0,s,Q_0}f(y) + \eta\}| < s|Q_0|\,.\]
Thus, picking $\eta_k \to 0^+$, by continuity from below it readily follows that
\begin{equation}
 |\{x \in Q_0: |f(x) - m_f(t,Q_0)| > 2 \inf_{y \in Q_0}M^{\sharp}_{0,s,Q_0}f(y)\}| \le s|Q_0|\,.
\end{equation}

Now let $f^0 = (f - m_f(t,Q_0) )\indicator{Q_0}$  and
\[\Omega^1 = \{x \in Q_0: m^{t,\Delta}_{Q_0}(f^0)(x) > 2\inf_{y \in Q_0}M^{\sharp}_{0,s,Q_0}f(y)\}\,.\]
Then by Theorem 2.1 in \cite{MedContOsc}, $E^1 \subset \Omega^1$ and $|\Omega^1| > 0$ as well.
Write $\Omega^1 = \bigcup_j Q^1_j$ where the $Q^1_j$
are nonoverlapping maximal dyadic subcubes of $Q_0$ such that
\begin{equation}
 |m_{f^0}(t,Q^1_j) | > 2\inf_{y \in Q_0}M^{\sharp}_{0,s,Q_0}f(y),\ {\text{and}}\
 |m_{f^0}(t,\widehat{Q^1_j}) | \le  2\inf_{y \in Q_0}M^{\sharp}_{0,s,Q_0}f(y).
\end{equation}
Since $m_{f^0}(t,Q_0) = 0$, $Q^1_j \ne  Q_0$ for any $j$.

Now since $t \ge 1/2$, from (1.10) in \cite{MedContOsc} it follows that
\begin{equation*}
2\inf_{y \in Q_0}M^{\sharp}_{0,s,Q_0}f(y) < |m_{f^0}(t,Q^1_j)| \le m_{|f^0|}(t,Q^1_j)\,,
\end{equation*}
and therefore by the definition of median
\begin{equation}
 |\{x \in Q^1_j: |f^0(x)| > 2\inf_{y \in Q_0}M^{\sharp}_{0,s,Q_0}\} | \ge (1-t)|Q^1_j|\,.
\end{equation}
When (2.4) is summed over $j$, we have by (2.2) that
\begin{align*}
(1-t)\sum_j |Q^1_j| &\le  \sum_j  |\{x \in Q^1_j: |f^0(x)| > 2\inf_{y \in Q_0}M^{\sharp}_{0,s,Q_0}f(y)\}|
\\
&\le |\{x \in Q_0: |f^0(x)| > 2\inf_{y \in Q_0}M^{\sharp}_{0,s,Q_0}f(y)\}| \le s|Q_0|\,,
\end{align*}
so that
\begin{equation}
\sum_j |Q^1_j| \le \frac{s}{1-t}\, |Q_0|\,,
\end{equation}
where by the choice of $s$ and $t$, $s/(1-t) < 1$.

Let $a^1_j = m_{f^0}(t,Q^1_j)$. By Lemma 4.3 in \cite{MedContOsc}  we see that
\begin{equation}
 | m_{f^0}(t,Q^1_j) - m_{f^0}(t,\widehat{Q^1_j})|
 \le 10n\inf_{y \in Q^1_j}M^{\sharp}_{0,s,\widehat{Q^1_j}}f(y) \,,
\end{equation}
and therefore by (2.3) and (2.6)
\begin{align}
|a^1_j| &\le |m_{f^0}(t,Q^1_j) - m_{f^0}(t, \widehat{Q^1_j})| + |m_{f_0}(t,\widehat{Q^1_j})|
\notag
\\
&\le 10n\inf_{y \in Q^1_j}M^{\sharp}_{0,s,\widehat{Q^1_j}}f(y) + 2\inf_{y \in Q_0}M^{\sharp}_{0,s,Q_0}f(y)\,.
\end{align}

The first iteration of the local median oscillation decomposition of $f$ when $|E^1| > 0$ is then as follows: for a.e. $\! x \in Q_0$, with $g^1 = f^0\indicator{Q_0 \setminus \Omega^1}$,
\[f^0(x) = g^1(x) + \sum_j a^1_j \indicator{Q^1_j}(x) + \sum_j   (f^0(x) - m_{f^0}(t,Q^1_j)   )\indicator{Q^1_j}(x)\,.\]
Note that $g^1$ has support off $\Omega^1$, and clearly for a.e. $\! x \in Q_0$,
\[|g^1(x)| \le  2\inf_{y \in Q_0}M^{\sharp}_{0,s,Q_0}f(y)\,.\]

Now focus on the second sum. Since $f^0(x) - m_{f^0}(t,Q) = f(x) - m_f(t,Q)$ for all cubes $Q$ and functions $f$ supported in $Q$, this sum equals
\[\sum_j    ( f(x) - m_f(t,Q^1_j)   )\indicator{Q^1_j}(x)\,.\]
The idea is to repeat the above argument for each of the functions
 $f^1_j = ( f - m_f(t,Q^1_j) )\indicator{Q^1_j}$, and so on.

We now describe the iteration. Assuming that $\{Q^{k-1}_j\}$ are the dyadic cubes
corresponding to the $(k-1)$st generation of subcubes of $Q_0$ obtained as above, let
\[f_j^{k-1} =  ( f - m_f(t,Q^{k-1}_j) )\indicator{Q^{k-1}_j}\,,\]
and
\[E^k_j = \{x \in Q^{k-1}_j: f^{k-1}_j(x) > 2\inf_{y \in Q^{k-1}_j}M^{\sharp}_{0,s,Q^{k-1}_j}f(y)\}\,.\]

If $|E^k_j| = 0$,  we write $s^k_j = f^{k-1}_j$ which satisfies
\begin{equation}
|s^k_j(x)| \le  2\inf_{y \in Q^{k-1}_j}M^{\sharp}_{0,s,Q^{k-1}_j}f(y)
\end{equation}
for a.e. $  x \in Q^{k-1}_j$. These are the ``$s$'' functions since the decomposition ``stops'' at $Q^{k-1}_j$; clearly $s^k_j$ has its support on $Q^{k-1}_j$, which contains no further subcubes of the decomposition.

If $|E^k_j| > 0$, we define
\[\Omega^k_j = \big\{x \in Q^{k-1}_j: m^{t,\Delta}_{Q^{k-1}_j}(f^{k-1}_j)(x)
 > 2\inf_{y \in Q^{k-1}_j}M^{\sharp}_{0,s,Q^{k-1}_j}f(y)\big\} \supset E^k_j\,.\]
Note that the $Q^{k-1}_j$, and thus the $\Omega^k_j$, are nonoverlapping. Then $|\Omega^k_j| > 0$ as well  and
\[\Omega^k_j = \bigcup_i Q^k_i\,,\]
where the $Q^k_i$'s are nonoverlapping maximal dyadic subcubes of $Q^{k-1}_j$ such that
\begin{align}
|m_{f^{k-1}_j}(t,Q^k_i)| > &2\inf_{y \in Q^{k-1}_j}M^{\sharp}_{0,s,Q^{k-1}_j}f(y),
\ {\text{and}}\notag\\
&|m_{f^{k-1}_j}(t,\widehat{Q^k_i})| \le 2\inf_{y \in Q^{k-1}_j}M^{\sharp}_{0,s,Q^{k-1}_j}f(y) \,.
\end{align}
Then define
\[\Omega^k = \bigcup_j \Omega^k_j\,.\]

Let
\begin{equation*}
a^{k,j}_i = m_{f^{k-1}_j}(t,Q^k_i)\,,
\end{equation*}
 and note that by (2.6) and (2.9)
\begin{align}
|a^{k,j}_i| &\le  |m_{f^{k-1}_j}(t,Q^k_i) - m_{f^{k-1}_j}(t,\widehat{Q^k_i})|
+ |m_{f^{k-1}_j}(t,\widehat{Q^k_i})|
\notag
\\
&\le 10n \inf_{y \in Q^k_i} M^{\sharp}_{0,s,\widehat{Q^k_i}}f(y)
+ 2\inf_{y \in Q^{k-1}_j}M^{\sharp}_{0,s,Q^{k-1}_j}f(y)\,.
\end{align}

We then have
\[f^{k-1}_j(x) = g_j^k(x) + \sum_i \alpha^{k,j}_i\indicator{Q^k_i}(x) + \sum_i  (f(x) - m_f(t,Q^k_i) )\indicator{Q^k_i}(x)\,,\]
for a.e. $\! x \in Q^{k-1}_j$, where $g^k_j = f^{k-1}_j \indicator{Q^{k-1}_j \setminus \Omega^k_j}$ is readily seen to satisfy
\begin{equation}
|g^k_j(x)| \le 2 \inf_{y \in Q^{k-1}_j}M^{\sharp}_{0,s,Q^{k-1}_j}f(y)
\end{equation}
for a.e. $\! x \in Q^{k-1}_j$. These are the ``$g$'' functions since the decomposition ``goes on'' or continues, into $Q^{k-1}_j$; $g^k_j$ has support on $Q^{k-1}_j$ away from $\Omega^k_j$, which are the next subcubes in the decomposition.

We separate the $Q^{k-1}_j$ into two families. One family, indexed by $I^k_1$, contains those cubes where the decomposition stops, and the other, indexed by $I^k_2$, where it continues. Specifically, let
\[I^k_1 = \{j : \Omega^k \cap Q^{k-1}_j = \emptyset\}, \quad I^k_2 = \{j: \Omega^k \cap Q^{k-1}_j \neq \emptyset\}.\]
Now we group the $Q^k_i$ based on which $Q^{k-1}_j$ contains them:  if $j \in I^k_2$, let
\[J_{j}^k = \{i : Q^k_i \subset Q^{k-1}_j\}.\]
These definitions
give that
\[\Omega_{j}^k = \bigcup_{i \in J_{j}^k}Q^k_i\,.\]

Note that, as in (2.5),
\begin{equation}|\Omega^k_j \cap Q^{k-1}_j| = \sum_{i \in J_{j}^k}|Q^k_i| \le \Big( \frac{s}{1-t}\Big) |Q^{k-1}_j|
\end{equation}
so that
\begin{align}
|\Omega^k| &= \sum_j |\Omega^k_j \cap Q^{k-1}_j| \le  \Big( \frac{s}{1-t} \Big)\sum_j |Q^{k-1}_j|
\notag
\\
&= \Big( \frac{s}{1-t} \Big)|\Omega^{k-1}| \le  \Big( \frac{s}{1-t} \Big)^k|Q_0|\,.
\end{align}

In fact, we claim that for all $j$ and $1 \le v < k$,
\begin{equation}
|\Omega^k \cap Q^v_j| \le  \Big( \frac{s}{1-t} \Big)^{k-v}|Q^v_j|\,,
\end{equation}
an estimate that  is useful in what follows.

To see this, for a given $k$, let $1\le v\le k-1$; if $v=k-1$ the conclusion is
(2.12). Next, if $v=k-2$ note that
\begin{align*}
|\Omega^{k} \cap Q^{k-2}_j| &= \sum_{Q^{k}_l \subset Q^{k-2}_j}|Q^{k}_l|
= \sum_{Q^{k-1}_i \subset Q^{k-2}_j} \sum_{Q^{k}_l \subset Q^{k-1}_i}|Q^{k}_l|
\\
&= \sum_{Q^{k-1}_i \subset Q^{k-2}_j} |\Omega^{k}_i \cap Q^{k-1}_i|
\le \Big( \frac{s}{1-t} \Big) \sum_{Q^{k-1}_i \subset Q^{k-2}_j}|Q^{k-1}_i|
\\
&= \Big( \frac{s}{1-t} \Big)\, |\Omega^{k-1}_j \cap Q^{k-2}_j|
\le \Big( \frac{s}{1-t} \Big)^2 |Q^{k-2}_j|\,,
\end{align*}
where the   inequalities follow  by (2.12). Continuing recursively, we have (2.14).

The $k$th iteration of the local median oscillation decomposition of the function $f$ is as follows: for a.e. $\! x \in Q_0$,
\begin{align*}
 f(x) - m_f &(t,Q_0)\\
  &= \sum_{v = 1}^k \Big(\sum_{j \in I^v_1}s^v_j +
  \sum_{j \in I^v_2} g^v_j \Big) + \sum_{v = 1}^k \sum_{j \in I^v_2} \sum_{i \in J_{j}^v}a^{v,j}_i \indicator{Q^v_i}(x) + \psi^k(x)\,,
\end{align*}
where
\[\psi^k = \sum_{j \in I^k_2} \sum_{i \in J_{j}^k}   (f - m_f(t,Q^k_i) )\indicator{Q^k_i}\,.\]

Since $\psi^k$ is supported in $\Omega^k$, by (2.13) it readily follows that $\psi^k \to 0$ a.e.\! in $Q_0$ as $k \to \infty$, and therefore
\begin{align*}
f(x) - m_f(t,Q_0) &= \sum_{v = 1}^{\infty} \Big( \sum_{j \in I^v_1}s^v_j + \sum_{j \in I^v_2} g^v_j \Big) + \sum_{v = 1}^{\infty}\sum_{j \in I^v_2} \sum_{i \in J^v_j} a^{v,j}_i \indicator{Q^v_j}(x)
\\
&=S_1(x) + S_2(x) \,,
\end{align*}
say.

In order to bound $|f(x) - m_f(t,Q_0)|$, consider first $S_1$. Of course,
for all $v$ and $j$ the $s^v_j$'s have nonoverlapping support. This is also true for the $g^v_j$.
Furthermore, the support of any $g^v_j$ is nonoverlapping with that of any $s^v_j$. So for every $v$,
$j$, and a.e. $\! x \in Q_0$, by (2.8) and (2.11) it readily follows that
\begin{align}
\Big| \sum_{v=1}^{\infty} \Big(\sum_{j \in I^v_1}s^v_j &+ \sum_{j \in I^v_2} g^v_j \Big) \Big|
\notag
\\
 &\le \max \Big\{ \sup_{j \in I^v_1} \big\| f^{v-1}_j  \big\|_{L^{\infty}}, \; \sup_{j \in I^v_2} \big\| f^{v-1}_j\indicator{Q^{v-1}_j \setminus \Omega_{j}^v} \big\|_{L^{\infty}} \Big\}
\notag
\\
&\le \max \Big\{ \sup_{j \in I^v_1} \Big(2\inf_{y \in Q^{v-1}_j} M^{\sharp}_{0,s,Q^{v-1}_j}f(y) \Big),
\notag
\\
&\qquad\qquad\qquad\qquad \sup_{j \in I^v_2} \Big(2\inf_{y \in Q^{v-1}_j}M^{\sharp}_{0,s,Q^{v-1}_j}f(y) \Big) \Big\}
\notag
\\
&\le 2\,M^{\sharp}_{0,s,Q_0}f(x)\,.
\end{align}

We consider $S_2$ next. The summand for $v = 1$ is distinguished, so we deal with it separately. By (2.7) above,
\begin{align}
\Big |\sum_j & a_j^1 \indicator{Q^1_j}(x) \Big |\notag\\
 &\le \sum_j |a_j^1| \indicator{Q^1_j}(x)
\notag
\\
&\quad\le \sum_j \Big (10n\inf_{y \in Q^1_j}M^{\sharp}_{0,s,\widehat{Q^1_j}}f(y) + 2\inf_{y \in Q_0}M^{\sharp}_{0,s,Q_0}f(y) \Big )\indicator{Q^1_j}(x)
\notag
\\
&\quad\quad \le \sum_j \Big(10n \inf_{y \in Q^1_j}M^{\sharp}_{0,s,\widehat{Q^1_j}}f(y) \Big)\indicator{Q^1_j}(x) + 2\inf_{y \in Q_0}M^{\sharp}_{0,s,Q_0}f(y)\,.
\end{align}

 As for the other terms of the sum, by (2.10) we have
\begin{align}
\Big |& \sum_{v = 2}^{\infty}  \sum_{j \in I^v_2} \sum_{i \in J^v_j}a^{v,j}_i \indicator{Q^v_i}(x)\Big |\notag\\
& \le \sum_{v = 2}^{\infty} \sum_{j \in I^v_2} \sum_{i \in J_{j}^v}|a^{v,j}_i| \indicator{Q^v_i}(x)
\notag
\\
&\le \sum_{v = 2}^{\infty} \sum_{j \in I^v_2} \sum_{i \in J_{j}^v}\Big (10n \inf_{y \in Q^v_i} M^{\sharp}_{0,s,\widehat{Q^v_i}}f(y) + 2\inf_{y \in Q^{v-1}_j}M^{\sharp}_{0,s,Q^{v-1}_j}f(y) \Big )\indicator{Q^v_i}(x)
\notag
\\
&\le  \sum_{v = 2}^{\infty} \sum_{j \in I^v_2}\sum_{i \in J^v_j} \Big(10n \inf_{y \in Q^v_i}M^{\sharp}_{0,s,\widehat{Q^v_i}}f(y)
\Big )\indicator{Q^v_i}(x) \notag
\\
&\qquad   +\sum_{v = 2}^{\infty}\sum_{j \in I^v_2} \Big(2\inf_{y \in Q^{v-1}_j}M^{\sharp}_{0,s,Q^{v-1}_j}f(y) \Big)
\indicator{Q^{v-1}_j}(x)\,.
\end{align}

We combine (2.16) and (2.17) and note that since the sum is infinite and the families $I^v_2$ are nested,
\begin{align}
\Big |\sum_{v = 1}^{\infty} &\sum_{j \in I^v_2} \sum_{i \in J^v_j}a^{v,j}_i \indicator{Q^v_i}(x) \Big |\notag \\
&\le \sum_{v=1}^{\infty}\sum_{j \in I^v_2} \sum_{i \in J^v_j} \Big(10n\inf_{y \in Q^v_i}M^{\sharp}_{0,s,\widehat{Q^v_i}}f(y) \Big) \indicator{Q^v_i}(x)
\notag
\\
&\qquad\qquad+ \sum_{v=1}^{\infty}\sum_{j \in I^v_2}\Big(2\inf_{y \in Q^{v-1}_j}M^{\sharp}_{0,s,Q^{v-1}_j}f(y) \Big)\indicator{Q^{v-1}_j}(x)
\notag
\\
&\le \sum_{v=1}^{\infty}\sum_{j \in I^v_2} \Big(10n\inf_{y \in Q^v_j}M^{\sharp}_{0,s,\widehat{Q^v_j}}f(y)
+ 2 \inf_{y \in Q^v_j}M^{\sharp}_{0,s,Q^v_j}f(y)\Big) \indicator{Q^v_j}(x)
\notag
\\
&\qquad\qquad + 2\inf_{y \in Q_0}M^{\sharp}_{0,s,Q_0}f(y)\,.
\end{align}

Combining (2.15) and (2.18), finally we get that for a.e. $\! x \in Q_0$,
\begin{align*}
|f(x) - & m_f(t,Q_0)| \le 4\, M^{\sharp}_{0,s,Q_0}f(x)
\\
&+ \sum_{v=1}^{\infty}\sum_{j \in I^v_2} \Big(10n\inf_{y \in Q^v_j}M^{\sharp}_{0,s,\widehat{Q^v_j}}f(y)
+ 2 \inf_{y \in Q^v_j}M^{\sharp}_{0,s,Q^v_j}f(y)\Big) \indicator{Q^v_j}(x)\,,
\end{align*}
and we have finished.
\end{proof}

\section{Weighted local mean estimates for local maximal functions}

In this section we consider the control of a weighted local mean of a function by
the weighted local mean of its local maximal function. In $\R^n$, for the sharp maximal function, in the unweighted case
this result was first established  by Fefferman and Stein \cite{FSActa},
and in the weighted case by several authors, including Fujii \cite{Fujii1989}.
For the local sharp maximal function and $A_\infty$ weights $w$, it follows from the fact that
 there exists a constant $0 < s_1  < 1$ with  the  following  property:
 given  $0<s\le s_1$, there exist constants $c, c_1$ such  that  for all cubes $Q$,
\[ w \big (\{ y\in Q: |f(y)-m_f(s,Q)|>\lambda, M^\sharp_{0,s}f(y)<\alpha\} \big )\le c\, e^{-c_1\lambda/\alpha}\,w(Q)\,,\]
for all $\lambda, \alpha> 0$. This is proved in Chapter III of \cite{TorchinskyStromberg}.

We are interested in the weighted local version of these results involving weights that are not necessarily $A_\infty$.
\begin{mydef}
We say that the weights $(w,v)$ satisfy condition $F$ if
there exist positive constants $c_1, \alpha, \beta$ with $0 < \alpha <
 1$, such that for any cube $Q$ and  measurable subset $E$ of $Q$ with $|E| \le  \alpha |Q|$,
\begin{equation}
\int_E w(x)\,dx \le  c_1 \Big ( \frac{|E|}{|Q|} \Big )^{\beta}
\int_{Q \setminus E}v(x)\, dx\,.
\end{equation}
\end{mydef}

Fujii observed that if $(w,v)$ satisfy condition $F$, then
$w(x)\le c\, v(x) $ a.e., and that for $w = v$, (3.1) is equivalent to the $A_{\infty}$ condition for $w$.
He also gave a simple example of a pair $(w , v)$ that satisfy condition $F$
so that neither of them is an $A_\infty$ weight   and no $A_\infty$ weight
can be inserted between them:  let $w(x) = 0$ if $0 < x < 1$ and $w(x) = 1$ otherwise, and
$v(x) = 0$ if $1/3< x < 2/3$ and $v(x) = 1$ otherwise \cite{Fujii1989}.

Along similar lines, if $w$ is in weak $A_\infty$, i.e.,
there exist positive numbers $c, \beta$ such that
for any cube $Q$ and  measurable subset $E$ of $Q$,
\[ w(E)\le c \Big ( \frac{|E|}{|Q|} \Big )^{\beta}\,w (2Q)\,,
\]
  a simple computation gives that $(w, Mw)$ satisfy condition $F$.
In fact, by an observation of Sawyer \cite{SawProc} this
also follows from the next example.

 We say that a weight $w$ is in the Muckenhoupt class  $C_p$ if there exist positive constants $\beta, c$ such that
\[ \int_E w(x)\,dx \le  c\,  \Big(\frac{|E|}{|Q|}\Big)^\beta \int_{\R^n}   M(\indicator{Q})(x)^p\, w( x )\, dx \]
whenever $E$ is a subset of a cube $Q\subset \R^n$; clearly $A_\infty \subset  C_p$, $1<p<\infty$,
 but $C_p$ contains weights not in $A_\infty$, as the
example $w(x)=\indicator{[0,\infty)}(x)v(x)$ with $v\in A_\infty$ of the line shows.
Now, $C_p$ is necessary for the integral inequality (1.1) to hold with $\Phi(t)=t^p$, $1<p<\infty$,
and  $C_q$ with $q>p$ is sufficient for (1.1) to hold (\,\cite{Muck, SawProc}, and \cite{Yab}.)

Note that for a fixed $0<\alpha<1$, we have  $M(\indicator{Q})(x)
 \le c\, M(\indicator{{Q\setminus E}})(x)$,  where $E\subset Q$ is such that  $|E|\le \alpha |Q|$
and $c$ depends on $\alpha$ but is independent of $E$ and $Q$. Then by the  Fefferman-Stein maximal inequality,
\begin{align*}
\int_{\R^n} M(\indicator{Q})(x)^p\, w( x )\, dx &\le c\int_{\R^n} M(\indicator{Q\setminus E})(x)^p\, w( x )\, dx\\
&  \le c \int_{Q\setminus E} M w( x )\, dx\,.
\end{align*}
Thus, if $w$ satisfies the condition $C_p$ for some $p$, $1<p<\infty$,
\[ \int_{E}  w(x)\,dx \le c\, \Big(\frac{|E|}{|Q|}\Big)^\beta\,    \int_{Q\setminus E} M w( x )\, dx
\]
and $(w, Mw)$ satisfy condition $F$.

A word of caution: at the end of this section we show indirectly
that for some weight $w$,  $(w,Mw)$ do not satisfy
condition $F$. Along these lines, for a Young function $A$  let
\[\|f\|_{L^A(Q)}= \inf \Big\{\lambda>0: \frac1{|Q|}\int_Q A\Big( \frac{|f(y)|}{\lambda}\Big)\,dy\le 1\Big\}\,,\]
and
\[ M_A f(x)=\sup_{x\in Q} \|f\|_{L^A(Q)}\,.\]
It then readily follows that for a weight $w$, $M_A w\in A_1$ if
 \[\int_0^t \frac{A(s)}{s^2}\,d s\le c\, \frac{A(t)}{t}\,.\]
Hence,
for any weight $w$ and $1<r<\infty$,  $  ( w, M_r w   )$ satisfy condition $F$.

On the other hand, for an integer $k=0,1,\ldots$, let   $A_k(t)=t\log^{k}(1+t)$. Then, if
$M^{k+1}$ denotes  the $k+1$ composition  of the Hardy-Littlewood maximal function operator
with itself,
 $M^{k+1}$ is pointwise comparable to the maximal operator $M_{A_k}$,  \cite{Compo},
and  by the comments
after Theorem 5.4,  for every $k$ there exists a weight $w$ such that  $(w,M_{A_k}w)$ do not
satisfy condition $F$. In particular, for such a weight $w$, $M_{A_k}w \notin A_\infty$.

Two remarks are in order before  we proceed to prove the main result in this section. First, the choice of the parameters
$s$ and $t$   in (3.2) below remains fixed throughout the paper unless otherwise noted, and second,
note that in the proof below the constant is linear with respect to the constant $c_1$ of the
weights that satisfy condition $ F$, and in particular, linear  in the $A_\infty$ norm of $w$.

\begin{theorem}
Let $\Phi$ satisfy condition $C$ with doubling constant $c_0$, $(w,v)$  weights on $\mathbb{R}^n$ satisfying condition $F$
with constants $\beta, c_1$, and pick $s,t$ such that
$0 < s \le 1/2$,  $1/2 < t < 1-s$, and
\begin{equation}
c_0 \Big (\frac{s}{1-t} \Big )^{\beta} < 1\,.
\end{equation}

Then for any measurable function $f$ and a cube $Q_0\subset \R^n$,
with a constant $c$ independent of $\Phi, Q_0,$ and $f$,
\begin{equation}
\int_{Q_0}\Phi  (|f(x) - m_f(t,Q_0)|  )\, w(x)\,dx \le  c \int_{Q_0}\Phi  ( M^{\sharp}_{0,s,Q_0}f(x)   )\,v(x)\,dx\,.
\end{equation}

Furthermore, if $f$ is such that $\lim_{ Q  \to \R^n} m_f(t,Q) = 0$,
  we also have
\begin{equation}
\int_{\R^n} \Phi  (|f(x)|  )\,w(x)\,dx \le c
\int_{\R^n} \Phi (M^{\sharp}_{0,s}f(x) )\, v(x)\,dx\,.
\end{equation}
\end{theorem}

\begin{proof}
Fix a cube $Q_0$. Then by Theorem 2.1, for a.e.  $x \in Q_0$,
\[|f(x) - m_f(t,Q_0)| \le 4M^{\sharp}_{0,s,Q_0}f(x) + \sum_{v=1}^{\infty}\sum_{j \in I^v_2}a^v_j \indicator{Q^v_j}(x)\,,\]
where by (2.1),
\[a^v_j \le (10n+2)\inf_{y \in Q^v_j} M^{\sharp}_{0,s,\widehat{Q^v_j}}f(y) \leq (10n+2)\inf_{y \in Q^v_j}M^{\sharp}_{0,s,Q_0}f(y).\]

Then
\begin{align*}
 \Phi(|f(x) - m_f(t,Q_0)|)
 &\le \Phi \Big(  4M^{\sharp}_{0,s,Q_0}f(x) + \sum_{v=1}^{\infty}\sum_{j \in I^v_2} a^v_j
\indicator{Q^v_j}(x)  \Big )
\\
&\le  c_0^3   \Phi (M^{\sharp}_{0,s,Q_0}f(x) ) + c_0 \Phi \Big (
\sum_{v=1}^{\infty}\sum_{j \in I^v_2} a^v_j \indicator{Q^v_j}(x) \Big ) \,,
\end{align*}
and therefore,
 \begin{align*}
\int_{Q_0}  \Phi (|f(x) - & m_f(t,Q_0)| )\, w(x)\, dx
\\
&\le c \int_{Q_0}  \Phi (M^{\sharp}_{0,s,Q_0}f(x) )\, w(x)\,dx\\
&\qquad\qquad + c \int_{Q_0} \Phi \Big (
\sum_{v=1}^{\infty}\sum_{j \in I^v_2} a^v_j \indicator{Q^v_j}(x) \Big )w(x)\, dx
\\
&= c\,I + c\, J,
\end{align*}
say. Now, since $w(x)\le c_3\, v(x)$ a.e., $I$ is of the right order.

As for $J$, since $\Phi(0) = 0$ and the $\Omega^k$ are nested,  the domain of integration extends to $\Omega^1 = \bigcup_{k=1}^\infty (\Omega^k\setminus
\Omega^{k+1})$. Then we may write $J=\sum_{k=1}^\infty J_k$ where
\[J_1= \int_{\Omega^1 \setminus \Omega^2} \Phi \Big ( \sum_{j \in I^1_2}
a^1_j \indicator{Q^1_j}(x) \Big )w(x)\,dx\]
and for $k\ge 2$, since only cubes of up to the $k$th generation enter in $\Omega^k \setminus \Omega^{k+1}$,
\[ J_k = \int_{\Omega^k \setminus
\Omega^{k+1}} \Phi \Big (\sum_{v=1}^k \sum_{j \in
I^v_2} a^v_j \indicator{Q^v_j}(x) \Big )w(x)\,dx\,.
\]

Focusing on the $J_k$ for $k \ge  2$, the integrand is bounded by
\begin{align*}
 &\Phi \Big  (\sum_{v=1}^k  \sum_{j \in I^v_2} a^v_j
 \indicator{Q^v_j}(x) \Big )\\
&\qquad\quad \le  \sum_{v=1}^k c_0^{k-v+1} \Phi \Big (\sum_{j \in
I^v_2} a^v_j \indicator{Q^v_j}(x) \Big )
\\
&\qquad\quad\qquad\quad
\le c_0  \Phi \Big (\sum_{j\in I^v_2} a^k_j
\indicator{Q^k_j}(x) \Big ) + \sum_{v=1}^{k-1} c_0^{k-v+1} \Phi \Big ( \sum_{j \in I^v_2} a^v_j \indicator{Q^v_j}(x)
\Big ) \,,
\end{align*}
and therefore  $J_k$ does not exceed
\begin{align*} & c_0 \Big(\int_{\Omega^k \setminus \Omega^{k+1}} \Phi \Big (\sum_{j\in I^v_2} a^k_j
\indicator{Q^k_j}(x) \Big ) \,w(x)\,dx\\
&\qquad\quad +
 \sum_{v=1}^{k-1} c_0^{k-v} \int_{\Omega_k} \Phi \Big ( \sum_{j \in I^v_2} a^v_j \indicator{Q^v_j}(x)
\Big )\,w(x)\,dx\Big) \\
&\qquad\quad\qquad\quad = c_0( J_k^1+J_k^2)\,,
\end{align*}
say.

Note that the $J_k^1$ and $J_1$ are essentially of the same form, and
with $c\le  c_0^{\log(10n + 2)}$ their total contribution is
\begin{align} J_1+\sum_{k=2}^\infty  J_k^1 &\le c_0 \sum_{k=1}^\infty \int_{\Omega^k \setminus \Omega^{k+1}} \Phi \Big (\sum_{j\in I^v_2} a^k_j
\indicator{Q^k_j}(x) \Big ) \,w(x)\,dx\notag
\\
&\le  c_0  \sum_{k=1}^{\infty} \sum_{j \in I^k_2}\int_{Q^k_j
\setminus \Omega^{k+1}} \Phi  ( a^k_j ) \, w(x)\, dx
\notag
\\
&\le c_0 c   \sum_{k=1}^{\infty} \sum_{j \in I^k_2}
\int_{Q^k_j \setminus \Omega^{k+1}}
\Phi ( M^{\sharp}_{0,s,Q_0}f(x) )\,w(x)\,dx
\notag
\\
&\le c_0 c  \int_{Q_0}
\Phi ( M^{\sharp}_{0,s,Q_0}f(x) )\, w(x)\,dx\,.
\notag
\end{align}

As for the  $J_k^2$, we claim that each $J_k^2$ satisfies,
for $1 \le v \le k-1$,
\begin{align*}\int_{\Omega^k}\Phi  \Big ( \sum_{j \in I^v_2} a^v_j \indicator{Q^v_j}(x) &\Big)\, w(x)\,dx\\
&\le c_1 \Big (\frac{s}{1-t} \Big )^{\beta(k-v)} \sum_{j \in I^v_2} \int_{Q^v_j \setminus \Omega^k}\Phi(a^v_j)\, v(x)\, dx\,.
\end{align*}

Indeed, since $\{Q^v_j\}_j$ are pairwise disjoint and $\Phi(0)= 0$, by condition $F$ we have
\begin{align}
&\int_{\Omega^k}  \Phi \Big ( \sum_{j \in I^v_2} a^v_j
\indicator{Q^v_j}(x) \Big )w(x)\,dx\notag\\
 &\quad\quad = \sum_{j \in I^v_2} \Phi (a^v_j )\int_{\Omega^k}
\indicator{Q^v_j}(x)\, w(x)\,dx\notag
\\
&\quad\quad\quad\quad = \sum_{j \in I^v_2} \Phi   ( a^v_j   ) \int_{\Omega^k \cap Q^v_j} w(x)\,dx\notag
\\
& \quad\quad\quad\quad\quad\quad \le c_1 \sum_{j \in I^v_2}  \Phi   (a^v_j   )  \Big (\frac{|\Omega^k \cap Q^v_j|}{|Q^v_j|}
\Big )^\beta \int_{Q^v_j \setminus
\Omega^k} v(x)\, dx\notag
\\
&\quad\quad\quad\quad\quad\quad\quad\quad
\le c_1 \, \Big (\frac{s}{1-t} \Big )^{\beta(k-v)}\sum_{j \in I^v_2}
\int_{Q^v_j \setminus \Omega^k} \Phi   ( a^v_j   )\,v(x)\, dx\,,
\end{align}
where the last inequality follows from (2.14) in Theorem 2.1.

Therefore, for $k\ge 2$, with
\[\alpha= c_0 \Big (\frac{s}{1-t} \Big )^{\beta }<1\,,\]
we have
\[J_k^2\le   c_1 \sum_{v=1}^{k-1} \alpha^{k-v} \sum_{j \in I^v_2} \int_{Q^v_j \setminus
\Omega^k} \Phi   ( a^v_j   ) \,v(x)\, dx\,,
\]
so that with $c\le  c_0^{\log(10n+2)}$,
\begin{align} \sum_{j \in I^v_2} \int_{Q^v_j \setminus
\Omega^k} \Phi   ( a^v_j  ) \, v(x)\, dx &\le c \sum_{j \in I^v_2}\int_{Q^v_j
\setminus \Omega^k} \Phi (
M^{\sharp}_{0,s,Q_0}f(x)  )\, v(x)\, dx
\notag
\\
&\le c \int_{\Omega^v
\setminus \Omega^k} \Phi (
M^{\sharp}_{0,s,Q_0}f(x)  )\,v(x)\, dx \,.
\end{align}

It only remains to bound  $\sum_{k=2}^\infty J_k^2$.
By (3.5) we have
\begin{align*} \sum_{k=2}^\infty J_k^2  &\le  c \sum_{k=2}^{\infty} \sum_{v=1}^{k-1} \alpha^{k-v}  \int_{\Omega^v
\setminus \Omega^k} \Phi (M^{\sharp}_{0,s,Q_0}f(x)  )\, v(x)\, dx
 \\
&=  c \sum_{v=1}^{\infty} \sum_{k=v+1}^{\infty} \alpha^{k-v}  \int_{\Omega^v
\setminus \Omega^k} \Phi (M^{\sharp}_{0,s,Q_0}f(x) )\,v(x)\, dx
\\
&=  c \sum_{v=1}^{\infty} \sum_{k=1}^{\infty} \alpha^{k}  \int_{\Omega^v
\setminus \Omega^{v+k}} \Phi (M^{\sharp}_{0,s,Q_0}f(x) )\,v(x)\, dx
\\
&=  c \sum_{k=1}^{\infty} \alpha^{k}    \int_{Q_0} \sum_{v=1}^{\infty}   \indicator{\Omega^v
\setminus \Omega^{v+k}}(x) \, \Phi  (M^{\sharp}_{0,s,Q_0}f(x)  )\, v(x)\, dx\,.
\end{align*}

Now, since the $\Omega^v$ are nested and contained in $Q_0$, and for fixed $k$, a set $\Omega^j \setminus \Omega^{j+k}$ overlaps at most $k$ of the other sets $\{\Omega^v \setminus \Omega^{v+k}\}_v$, $\sum_{v=1}^{\infty}   \indicator{\Omega^v
\setminus \Omega^{v+k}}(x)\le k \,\indicator{Q_0}(x)$, and
\begin{align}\sum_{k=2}^\infty J_k^2 &\le
c\Big( \sum_{k=1}^{\infty}  k\,\alpha^k  \Big) \int_{Q_0} \Phi ( M^{\sharp}_{0,s,Q_0}f(x) )\,v(x)\,dx\notag\\
&\le c \int_{Q_0}\Phi (M^{\sharp}_{0,s,Q_0}f(x) )\,v(x)\,dx\,.
\end{align}

We thus have
\[\int_{Q_0 } \Phi   (|f(x) - m_f(t,Q_0)|   )\, w(x)\,dx \le  c \int_{Q_0}
\Phi (M^{\sharp}_{0,s,Q_0}f(x) )\, v(x)\,dx\,.\]

Furthermore, if   $\lim_{Q_0 \to \R^n}m_{ f }(t,Q_0) = 0$,
by Fatou's lemma
\[\int_{\mathbb{R}^n} \Phi   (|f(x)|   )\,w(x)\,dx \le c \int_{\mathbb{R}^n}
\Phi ( M^{\sharp}_{0,s}f(x) )\, v(x)\,dx\,.\]

This completes the proof.
\end{proof}

A couple of comments. First,
 observe that  if the right-hand side of (3.4) above is finite, then
as in Chapter III of \cite{TorchinskyStromberg}, $\lim_{Q_0\to\R^n}m_f(t,Q_0)=m_f$ exists along a sequence of $Q_0\to\R^n$, and the conclusion then reads
\[\int_{\R^n}\Phi (|f(x) - m_f| )\, w(x)\,dx \le  c \int_{\R^n}\Phi   ( M^{\sharp}_{0,s}f(x) )\, v(x)\,dx\,.\]
As Lerner observed, $m_f = 0$ if $f^*(+\infty) = 0$,
where  $f^*$ denotes the nonincreasing rearrangement of $f$,
which in turn holds if and only if
$|\{x \in \R^n : |f(x)| > \alpha\}| < \infty$
for all $\alpha > 0$, \cite{L}.
In particular, this holds if the support of $f$ has finite measure or if $f$ is in weak-$L^p(\R^n)$ for
some $0<p<\infty$.

Second, (3.5), (3.6) and (3.7)  hold if (3.1) above is replaced by
\begin{equation*}
\int_E w(x)\,dx \le  c_1 \psi \Big ( \frac{|E|}{|Q|} \Big )
\int_{Q \setminus E}v(x)\, dx\,,
\end{equation*}
where $\sum_{k=1}^\infty k \, c_0^k \, \psi(\alpha^k)<\infty$.
Thus, the class of weights that satisfy condition $F$ could be extended to include these general
$\psi$'s as well.

Our next result, essentially due to Lerner, holds for concave $\Phi$, including
$\Phi(t)=t$, with $v=Mw$
on the right-hand side of (3.3), \cite{Lerner2010, LernerSummary}.

\begin{theorem}
Let $\Phi$ be a concave function with $\Phi(0) = 0$, $f$
a measurable function on $\R^n$, and  $0 < s < 1/2$ and $1/2 \le t < 1-s$.
 Then for any weight $w$ and cube $Q_0 \subset \mathbb{R}^n$,
\begin{equation*}
\int_{Q_0}\Phi  (|f(x) - m_f(t,Q_0)|   )\,w(x)\,dx
\le  c\int_{Q_0}\Phi (M^{\sharp}_{0,s,Q_0}f(x) )\, Mw (x)\,dx\,.
\end{equation*}
Furthermore, if $f$ is such that  $m_f(t,Q_0) \to 0$ as $Q_0 \to \mathbb{R}^n$, then
\begin{equation}
\int_{\mathbb{R}^n}\Phi(|f(x)| )\,w(x)\,dx \le c\int_{\mathbb{R}^n}\Phi (M^{\sharp}_{0,s}f(x) )\,Mw(x)\,dx\,.
\end{equation}
\end{theorem}

\begin{proof}
By Theorem 2.1 and the concavity of $\Phi$ we have
\begin{align*}
\int_{Q_0}\Phi &(  |f(x) - m_f(t,Q_0)|  )\,w(x)\,dx\\
& \le c\int_{Q_0}\Phi (M^{\sharp}_{0,s,Q_0}f(x) )\, w(x)\,dx
+c\sum_{v=1}^{\infty}\sum_{j \in I^v_2} \Phi (a^v_j )\int_{Q^v_j}\,w(x)\,dx\,,
\end{align*}
where
\[a^v_j \le  (10n+2)\inf_{y \in Q^v_j} M^{\sharp}_{0,s,Q_0}f(y)\,.\]

Now, since $w(x)\le Mw(x)$, the integral term is of the right order. Next, by construction,
the $Q^v_j$ are nonoverlapping over fixed $v$, but  since each $Q^v_j$ is a subcube of some $Q^{v-1}_i$,
they are not nonoverlapping over all $v$.  So we define $F^v_j = Q^v_j \setminus \Omega^{v+1}$, which are pairwise disjoint over all $v$ and $j$.
 Note that by (2.13) we have
\begin{equation*}
|F^v_j| \ge  \Big (1 - \frac{s}{1-t} \Big )|Q^v_j| = c_{s,t}|Q^v_j|.
\end{equation*}

We then estimate
\begin{align*}
\Phi (a^v_j )\int_{Q^v_j}\,w(x)\,dx &\le c\ \Phi \Big (\inf_{y \in Q^v_j} M^{\sharp}_{0,s,Q_0}f(y) \Big ) \int_{Q^v_j} w(x)\,dx\\
& \le c\,c_{s,t}\,|F^v_j| \, \Phi \Big (\inf_{y \in Q^v_j}M^{\sharp}_{0,s,Q_0}f(y) \Big ) \frac{1}{|Q^v_j|}\int_{Q^v_j}w(x)\,dx
\\
&\le c  \Big( \int_{F^v_j} \Phi (M^{\sharp}_{0,s,Q_0}f(x))\,dx \Big)\, \inf_{y \in F^v_j}Mw(y)
\\
&\le c \int_{F^v_j} \Phi (M^{\sharp}_{0,s,Q_0}f(x) )\,Mw(x) \,dx\,.
\end{align*}
Then summing, by the disjointness of the $F^v_j$ this gives
\begin{align*}
\sum_{v=1}^{\infty} \sum_{j \in I^v_2} \inf_{y \in Q^v_j} M^{\sharp}_{0,s,Q_0} & f(y)  \int_{Q^v_j}w(x)\,dx
\\
&\le c \int_{Q_0} \Phi (M^{\sharp}_{0,s,Q_0}f(x) ) \, Mw(x)\,dx\,,
\end{align*}
and the desired conclusion follows in this case.

Finally, if $m_{f}(t,Q_0) \to 0$ as $Q_0 \to \R^n$,  Fatou's lemma gives (3.8).
\end{proof}

Now, for an arbitrary weight $w$, note that (3.8) cannot hold for arbitrary $\Phi$.
Indeed, suppose that it holds for $\Phi(t) = t^p$ for some $p>2$. Then if $f = Tg$, where $T$ is a singular
integral operator, from (3.8), Theorem 4.1 below, and
the Fefferman-Stein maximal inequality it readily follows that
\[\int_{\R^n} |Tg(x)|^p\, w(x)\, dx \leq c \int_{\R^n}  |g(x)|^p \,M_{L\log L}w(x) \, dx,\]
which contradicts    (1.11).
Thus for arbitrary $w$, $(w,Mw)$ gives (3.8) for some but not all $\Phi$,
and therefore for some $w$,  $(w, Mw)$ do not satisfy $F$.

\section{Pointwise Inequalities Revisited}
 We prove here a local version of the estimate
 $M^{\sharp}_{0,s}(Tf)(x) \le  c\,Mf(x)$ for Calder\'{o}n-Zygmund singular integral operators established in \cite{JawerthTorchinsky}.
 We also recast similar estimates with $T$ replaced by  a
 singular integral operator with kernel satisfying   H{\"o}rmander-type conditions
 and an integral operator with a homogenous kernel, and $M$ by an appropriate maximal function $M_T$.

We start with the singular integral case. First an observation of a geometric nature: there exists
a dimensional constant $c_n$ such that for every cube $Q $ in $\R^n$, if $x,x'\in Q$ and
  $y\notin 2^m Q$  for some $m\ge 1$, then
\begin{equation}  \frac{|x - x'|}{|x - y|} \le c_n \,2^{-m}\,.
\end{equation}

We then have,

\begin{theorem}
 Let $T$ be a singular integral operator defined by
\begin{equation}
Tf(x) = {\rm p.v.} \int_{\mathbb{R}^n} k(x,y)\,f(y)\,dy
\end{equation}
such that
\begin{enumerate}
\item[\rm (1)]
for some $c> 0$, $k(x,y)$ satisfies
\[ | k(x,y) - k(x',y)| \le c\,\frac1{|x-y|^n}\,  \omega\Big(\frac{|x - x'|}{|x - y|}\Big)\]
     whenever $x,x'\in Q$ and $y\in (2Q)^c$ for any cube $Q$,  where $\omega(t)$ is a nondecreasing function
     on $(0,\infty)$ such that
     \[\int_0^1 \omega(c_n t)  \, \frac{dt}{t}<\infty\,;\]
     and
\item[\rm (2)] for some $1\le r<\infty$, $T$ is of weak-type $(r,r)$.
\end{enumerate}

Then for  $0 < s \le  1/2$, any cube $Q_0$, and $x\in Q_0$,
\begin{equation}
M^{\sharp}_{0,s,Q_0}(Tf )(x) \le c \sup_{x \in Q, Q \subset Q_0} \inf_{y \in Q} M_rf(y)\,.
\end{equation}

Moreover, if we also have that $T(1)=0$ and
 \[\int_0^1 \omega(c_n t) \ln(1/t)\, \frac{dt}{t}<\infty\,,\]
then
\begin{equation*}
M^{\sharp}_{0,s,Q_0}(Tf)(x) \le  c \sup_{x \in Q, Q \subset Q_0} \inf _{y \in Q} M^{\sharp}_rf(y)\,.
\end{equation*}

In particular, if $Q_0 = \mathbb{R}^n$, then
\[M^{\sharp}_{0,s}(Tf)(x) \le c\, M_rf(x)\,,\quad {\text{and}}\quad
  M^{\sharp}_{0,s}(Tf)(x)  \le c\, M^{\sharp}_rf(x)\,,\]
 respectively.
\end{theorem}

\begin{proof}
We consider the case when $T(1) =0$ first. Fix a cube $Q_0 \subset \mathbb{R}^n$  and take $x \in Q_0$.
Let $Q \subset Q_0$ be a cube containing $x$ with center $x_Q$ and sidelength $l_Q$.
Let $1/2 \le  t \le  1-s$,  $f_1 =   (f - m_f(t,Q)  ) \indicator{2Q}$, and $f_2 =   (f - m_f(t,Q)   )\indicator{(2Q)^c}$.
Then by the linearity of $T$, $Tf(z)- Tf_2(x_Q)= Tf_1(z)+ Tf_2(z)-Tf_2(x_Q)$ for $z\in Q$.

We claim that there exist constants $c_1,c_2>0$ independent of $f$ and $Q$ such that
\begin{equation}
|\{z \in Q: |Tf_1(z)| >  c_1 \inf_{y \in Q}M^{\sharp}_r f(y)\}| < s\,|Q|\,,
\end{equation}
and
\begin{equation}
 \|Tf_2 - Tf_2(x_Q)  \|_{L^{\infty}(Q)} \le  c_{2} \inf_{y \in Q}M^{\sharp}_r f(y)\,.
\end{equation}

We prove (4.5) first. For any $z \in Q$, by (4.1)
\begin{align}
|Tf_2(z) &- Tf_2(x_Q)|\notag\\
 &\le  \int_{(2Q)^c}|k(z,y) - k(x_Q,y)|\, |f(y) - m_f(t,Q)|\,dy
\notag
\\
&\le c \int_{(2Q)^c} \frac1{|z-y|^n}\, \omega\Big(\frac{|x_Q - z|}{| y-z|}\Big) \, |f(y) - m_f(t,Q)|\,dy
\notag
\\
&= c \sum_{m=1}^{\infty} \int_{2^{m+1}Q\setminus 2^{m}Q}  \frac1{|z-y|^n}\, \omega\Big(\frac{|x_Q - z|}{| y-z|}\Big)\,
|f(y) - m_f(t,Q)|\, dy
\notag
\\
&\le c \sum_{m=1}^{\infty} \omega(c_n/ 2^{m})\, \frac{1}{|2^m Q|}\int_{2^mQ} |f(y) - m_f(t,Q)|\,dy\,.
\end{align}

It readily follows from Proposition 1.1 in \cite{MedContOsc} that for any cube $Q'$,
\[  | m_f (t,Q')-f_{Q'}\,|\le c\,\frac{1}{|Q'|}\,\int_{Q'}  |f(y)-f_{Q'}| \,dy\,,
\]
and consequently, with $c=c_s$,
\begin{align}
\frac{1}{|Q'|}\int_{Q'} &|f(y) - m_f(t,Q')|\, dy\notag\\
 &\le
\frac{1}{|Q'|}\int_{Q'} |f(y) -  f_{Q'}|\, dy + |f_Q- m_f(t,Q')|\notag
\\
&\qquad\qquad \le c\,\frac{1}{|Q'|}\,\int_{Q'}  |f(y)-f_{Q'}| \,dy\notag,
\\
&\qquad\qquad\qquad\qquad \le c  \inf_{y \in Q'}M^{\sharp}f(y)\,.
\end{align}

Also, from (4.7) and the triangle inequality we have
\begin{equation}
 |m_f(t,2 \,Q') - m_f(t, Q') | \le c \, \inf_{y \in Q'}M^{\sharp}f(y)\,.
\end{equation}
Then (4.7) and (4.8) give that
\begin{align}
\int_{2^m Q} &|f(y)  - m_f(t,Q)|\,dy\notag\\
 &\le  \int_{2^m Q} |f(y) - m_f(t,2^m Q)|\, dy
\notag
\\
&\qquad\qquad + \int_{2^m Q} \sum_{j=1}^m |m_f(t,2^j Q) - m_f(t, 2^{j-1}Q)|\, dy
\notag
\\
&\qquad \le  c\, |2^{m} Q|\,    \inf_{y \in Q}M^{\sharp}f(y)
 + c\, |2^{m} Q|\,   \sum_{j=1}^m \inf_{y \in 2^{j-1} Q} M^{\sharp}f(y)
\notag
\\
&\qquad\qquad \le  c\, |2^{m} Q|\,  (1+m) \inf_{y \in Q}M^{\sharp}f(y)\,.
\end{align}

Using (4.9), we   bound (4.6) as
\begin{align*}
|Tf_2(z) - Tf_2(x_Q)| &\le  c\, \Big( \sum_{m=1}^{\infty} (1+m)\, \omega(c_n/2^{m})\Big) \inf_{y \in Q}M^{\sharp}f(y)
\\
&\le c \, \Big(\int_0^1 \omega(c_n t)\ln(1/t)\, \frac{dt}{t}\Big) \inf_{y \in Q}M^{\sharp}f(y)\,,
\end{align*}
and so
\begin{equation*}
 \|Tf_2 - Tf_2(x_Q) \|_{L^{\infty}(Q)} \le  c_{2} \inf_{y \in Q}M^{\sharp}f(y) \le  c_{2} \inf_{y \in Q}M^{\sharp}_rf(y)\,.
\end{equation*}

As for (4.4), since $T$ is of weak-type $(r,r)$, by (4.7)
 and (4.8) we have that for any $\lambda > 0$,
\begin{align}
{\lambda^r} |\{z &\in Q  : |Tf_1(z)| > \lambda\}|\notag\\
  &\le  {c} \int_{2Q}|f(y) - m_f(t,Q)|^r\,dy
\notag
\\
&\qquad \le {c}\int_{2Q}|f(y) - m_f(t,2Q)|^r\,dy
\notag
\\
&\qquad\qquad\qquad  + c\,
 |m_f(t,Q) - m_f(t,2Q)|^r |2Q|
\notag
\\
&\qquad\qquad \le  c \, \inf_{y \in 2Q} M^{\sharp}_r f(y)^r {|Q|}+  c\,   \inf_{y \in Q}M^{\sharp} f(y)^r{|Q|}
\notag
\\
&\qquad\qquad\qquad \le c \,   \inf_{y \in Q}M^{\sharp}_rf(y)^r{|Q|}\,,
\notag
\end{align}
     and (4.4) follows by picking $\lambda =c_1 \inf_{y\in Q} M^\sharp_r f(y)$
for an appropriately chosen $c_1$.

 Then, with $c > \max \{c_1, c_2\}$, (4.4) and (4.5) give
\begin{align*}
|\{z \in Q&: |Tf(z) - Tf_2(x_Q)| > 2c \inf_{y \in Q}M^{\sharp}_r f(y)\}|
\\
&\le |\{z \in Q: |Tf_2(z) - Tf_2(x_Q)| > c_2 \inf_{y \in Q}M^{\sharp}_r f(y)\}|
\\
&\qquad\qquad + |\{z \in Q: |Tf_1(z)| > c_1 \inf_{y \in Q}M^{\sharp}_rf(y)\}|
\\
&<{s}|Q|.
\end{align*}
Whence
\[\inf_{c'} \; \inf  \{\alpha \ge  0: |\{z \in Q: |Tf(z) - {c'}| > \alpha\}| < s|Q| \}\le c \inf_{y \in Q}M^{\sharp}_r f(y),\]
and consequently, since this holds for all $Q\subset Q_0$, $x\in Q$,
\[M^{\sharp}_{0,s,Q_0}Tf(x)\le c \sup_{ {x\in Q,  Q\subset Q_0}} \inf_{y \in Q}M^{\sharp}_rf(y)\,.\]

To prove the case where $T(1) \ne  0$, let $f_1 = f\indicator{2Q}$ and $f_2 = f\indicator{(2Q)^c}$, and proceed as above.
Then, for any $z \in Q$, as in the proof of (4.6),
\begin{align}
|Tf_2(z) - Tf_2(x_Q)| &\le
 c \int_{(2Q)^c}  \frac1{|z-y|^n} \, \omega\Big(\frac{|x_Q - z|}{| y-z|}\Big)  \,|f(y)  |\, dy
\notag
\\
&\le c \sum_{m=1}^{\infty} \omega (c_n/2^m )\, \frac{1}{|2^m Q|}\int_{2^mQ} |f(y)  |\,dy
\notag
\\&\le c \, \Big(\int_0^1 \omega(c_n t) \, \frac{dt}{t}\Big) \inf_{y \in Q}M f(y)
\notag
\\
&\le c    \inf_{y \in Q}M_rf(y)\,.
\end{align}

And, as in the proof of (4.4),
\begin{equation}
{\lambda^r} |\{z \in Q: |Tf_1(z)| > \lambda\}|  \le  {c} \int_{2Q}|f(y) |^r\,dy
\le c  \, \inf_{y \in Q}M_rf(y)^r {|Q|}\,.
\end{equation}
Then as before,
\[M^{\sharp}_{0,s,Q_0}Tf(x)\le c \sup_{ {x\in Q,  Q\subset Q_0}} \inf_{y \in Q}M_rf(y)\,.\]

The proof is thus complete.
\end{proof}

That  $M_r$ is relevant on the right-hand side of (4.3) for all $r$,
 $1\le r <\infty$, is clear from (4.11) above, and   is useful when $T$ is not known to
 be of weak-type $(1,1)$.  Also, there are   operators of weak-type ($1,1)$  where $M_r$ is   necessary on the right-hand side
of (4.10), and hence on the right-hand side of (4.3),  for $1<r<\infty$.
These are the Calder\'on-Zygmund  convolution operators   of Dini type, i.e.,  $k(x) = {\Omega(x')}/{|x|^n}$, $x \ne 0$, where
 $\Omega$ is a function on $S^{n-1}$ that satisfies $\int_{S^{n-1}}\Omega(x')\,dx'=0$ and
 an $L^q$-Dini condition  for some $1\le q\le \infty$, \cite{KurtzW}.
Because of their similarity with the
singular integral operators with kernels satisfying H\"ormander-type conditions
 considered in Theorem 4.3 below, the
analysis of this case is omitted.

Theorem 4.1 is the prototype of the following general principle.
\begin{theorem}
 Let $T$ be a linear operator with the following property:  There exists a mapping $M_T$ with the property that for every fixed cube $Q_0$, for any $Q \subset Q_0$, there exist $x_Q\in Q$ and constants  $c_1,c_2>0$
such that every $f$ in a dense class of functions of the domain of $T$ can be written as $f=f_1+f_2$ so that
\begin{equation}|\{z \in Q: |Tf_1(z)| > c_1 \inf_{y \in Q} M_Tf(y)\}| < s|Q|\,,
\end{equation}
and
\begin{equation}
 \|Tf_2 - Tf_2(x_Q) \|_{L^{\infty}(Q)} \le  c_2 \inf_{y \in Q} M_Tf(y)\,.
\end{equation}

 Then, there exists a constant $c$ independent of $f$ and $Q_0$ such that
\begin{equation*}
M^{\sharp}_{0,s,Q_0} (Tf )(x) \le  c \sup_{x \in Q, Q \subset Q_0} \inf _{y \in Q} M_T f(y)\,.
\end{equation*}

In particular, if $Q_0 = \mathbb{R}^n$,
\begin{equation*}
M^{\sharp}_{0,s} (Tf )(x) \le c\, M_Tf(x)\,.
\end{equation*}
\end{theorem}

We  exploit Theorem 4.2 -- whose proof is similar to that of Theorem 4.1, and is therefore omitted --
 in the results that follow.

Our first observation is that  singular integral operators with kernels
satisfying  H\"{o}rmander-type conditions similar to the  convolution operators considered by Lorente et al. \cite{Lorente}
satisfy local-type estimates with an appropriate $M_T$. $M_T$ could be, as we noted above, $M_r$, or more generally
$M_A$, where $A$ is a Young function;  we denote by $\overline A$ its conjugate function, given by
\[
\overline A(t)=\sup_{s>0}\, ( st - A(s) )\,.
\]

We then have,
\begin{theorem} Let $T$ be a Calder\'on-Zygmund singular integral operator of weak-type $(1,1)$ such that
 for a Young function $A$, every cube $Q$, and $u,v\in Q$,
\[  \sum^\infty_{m=1} |2^{m+1}Q| \, \| \indicator{2^{m+1}Q\setminus  2^{m}Q} ( k (u, \cdot)- k(v,\cdot)  ) \|_{L^{{A}}(2^{m+1}Q)}\le c_A<\infty\,.\]

Then, with $c$ independent of $x$, $Q_0$, and $f$,
\begin{equation*}
 M^{\sharp}_{0,s,Q_0}  (Tf  )(x) \le c \sup_{{x\in Q,\,  Q \subset Q_0}}\inf_{y \in Q}M_{\overline{A}}f(y)\,.
 \end{equation*}
\end{theorem}
\begin{proof}
Let $f=f_1+f_2$, where $f_1=f \indicator{2Q}$. Then, since $T$ is of weak-type $(1,1)$, as in the proof of (4.11),
\begin{equation*}
{\lambda} |\{z \in Q: |Tf_1(z)| > \lambda\}|  \le  {c}  \int_{2Q}|f(y) |\,dy
\le c \,    \inf_{y \in Q}Mf(y)\,{|Q|}\,,
\end{equation*}
and since $Mf(y)\le M_{\overline{A}}f(y)$ for all $y$, (4.12) holds for an appropriately chosen $c_1$.

Next,
\begin{align*} \int_{\R^n\setminus 2Q} |k(u,y) - k(v,y)|& \,|f(y)|\,dy\\
&= \sum^\infty_{m=1} \int_{2^{m+1}Q\setminus  2^{m}Q}  | k (u,y)- k(v,y)|\, |  f (y)|\, dy
\end{align*}
is bounded using H\"older's inequality  for $A$ and $\overline A$ by
\begin{align*} 2 \sum^\infty_{m=1} |2^{m+1}Q| \, \| \indicator{2^{m+1}Q\setminus  2^{m}Q} ( k (u, \cdot)&- k(v,\cdot) )  \|_{L^{{A}}(2^{m+1}Q)}\,
\|  f\|_{L^{\overline{A}}(2^{m+1}Q)}\\
&\le 2\,c_A\, \inf_{y\in Q} M_{\overline{A}}(f)\,.
\end{align*}
Hence
\begin{equation*}
|Tf_2(u)-Tf_2(v)|\
\le    2\, c_A \inf_{y\in Q} M_{\overline{A}} f(y)\,,
\end{equation*}
and therefore (4.13) holds  with  $c_2=2\, c_A$.
 The conclusion then follows from Theorem 4.2 with $M_T=M_{\overline A}\,$.
\end{proof}

The idea of the proof is essentially that of \cite{KurtzW} and \cite{Lorente}, where $T$ is assumed to be
of convolution type.  In that case, if $k$
satisfies the $L^{A}$-H\"{o}rmander condition for any Young function $A$,
 it also satisfies the usual   $L^{1}$-H\"{o}rmander condition,
and $T$ is of weak-type $(1, 1)$.

Finally, we consider the integral operators with  homogeneous kernels  defined    as follows \cite{RiverUrci}.
If $A_1, \ldots, A_m$ are invertible matrices such that $A_{k}-A_{k'}$ is invertible
for $k\ne k'$, $1\le k,k'\le m$, and , $\alpha_i > 0$ for all $i$ and  $\alpha_1 + \cdots + \alpha_m = n $,
then
\begin{equation}
T  f(x) = \int_{\R^n} |x - A_1y|^{-\alpha_1} \cdots |x - A_m y|^{-\alpha_m}f(y)\,dy\,.
\end{equation}

For these operators we have,
\begin{theorem}
For $T$ defined as in {\rm{(4.14)}}, any cube $Q_0 \subset \R^n$, and $x\in Q_0$, we have
\begin{equation*}
M^{\sharp}_{0,s,Q_0}  (Tf )(x)  \le c \sum_{i=1}^m \sup_{ {x\in Q \,,Q \subset Q_0}}\inf_{y \in  {Q}} M
f(A_i^{-1}y)
\,.
\end{equation*}
\end{theorem}

\begin{proof}
Let $Q_0$ be a fixed cube and $x\in Q_0$.
As in the proof of Theorem 2.1 in  \cite {RiverUrci},  for a cube $Q $ containing $x$
for an appropriate dimensional constant  $\lambda$
and $1\le i\le m$, let $ Q_i=A_i^{-1}(\lambda Q)$,
and put
\[ f_1(y) =f (y)\indicator{\bigcup_{i=1}^m  Q_i}(y)\,,\]
 and $f_2=f-f_1$.

First, by Theorem 3.2 in \cite{RiverUrci}, $T$ is of weak-type $(1,1)$.
Moreover,  since   by the inequality following (2.2) in the proof of Theorem 2.1 in \cite{RiverUrci},
for all integrable functions $g$,
\[  \sum_{i=1}^m  \int_{  Q_i}|g(y)|\,dy
\le c\, \Big(\sum_{i=1}^m \inf_{y\in Q} M  g(A_i^{-1}y)\Big) \,|Q|\,,
\]
taking  $\lambda > ({c}/{s}){\sum_{i=1}^m \inf_{y \in  {Q}}Mf(A_i^{-1}y)}$,
it follows that
\begin{equation*}
|\{y \in Q: |Tf_1(y)| > \lambda \}| < s|Q|.
\end{equation*}

And, concerning the $Tf_2$ term, for any $y \in Q$ we have
\[|Tf_2(y) - Tf_2(x_Q)| \le \int_{\R^n\setminus \bigcup_{1 \le i \le m} {Q}_i } |k(y,z) - k(x_Q,z)|\,|f(z)|\, dz\,.\]

Now, by breaking up $\R^n\setminus \bigcup_{1 \le i \le m} {Q}_i   $ into regions as   in (2.4)
in the proof of Theorem 2.1 of \cite{RiverUrci}, we have  that
\[ |Tf_2(y) - Tf_2(x_Q)|\le   c\,   \sum_{i=1}^m \inf_{y \in {Q}}Mf(A_i^{-1}y)\,.
\]

The  conclusion then follows as indicated in Theorem  4.2.
\end{proof}

To obtain   local weighted estimates one could of course
 use the full strength of the result in \cite{RiverUrci}, namely,
\[M^\sharp(|Tf|^\delta)(x)^{1/\delta}\le c\, \sum_{i=1}^m M  f(A_i^{-1} x)\,,
\]
where $0<\delta<1$.

\section{Local weighted estimates}
The  local estimates in the previous section can be used to express the local integral control of
$Tf$ in terms of $M_Tf$. Specifically, we have
\begin{theorem}
Let $T, M_T$ be  operators    such that the conditions of Theorem 4.2 hold.
 Then for any $\Phi$ satisfying condition $C$, cube $Q_0$, and weights $(w,v)$ satisfying condition $F$,
\[\int_{Q_0} \Phi  (|Tf(x) - m_{Tf}(t,Q_0)|  ) \,w(x)\,dx \le c \int_{Q_0}   \Phi  ( M_T f(x) )\, v(x)\, dx\,,\]
where $c$ is independent of $Q_0$ and $f$.

Furthermore, if $\lim_{Q_0 \to \R^n} m_{Tf}(t,Q_0) = 0$,
\[\int_{\mathbb{R}^n}\Phi    (|Tf (x)|   )\, w(x)\,dx \le c \int_{\mathbb{R}^n}\ \Phi  ( M_T f(x)  )\, v(x)\,dx\,.\]
\end{theorem}
\begin{proof}
The proof follows immediately from   Theorem 3.1 and Theorem 4.2. Thus, the result holds with $M_T=M_r$ or $M^\sharp_r$ for singular integral operators, $M_T=M_{\overline A}$ for  singular integral operators with kernels satisfying H\"ormander-type conditions,
and $M_Tf(x)=\sum_{i=1}^m Mf(A_i^{-1}x)$ for integral operators with homogeneous kernels.
\end{proof}

We   discuss now in some detail the case  $M_T=M_r$. Note that, in particular,
Theorem 5.1 (with $r=1$) gives (1.1)  as well as (1.3), which  are then one and the same result.

And, concerning (1.4), we have the following observation.
\begin{theorem} Let $T$ be an operator    such that the conditions of Theorem 4.2 hold with $M_T=M_r$, $1\le r<\infty$.
Then,   for $0<p<r$,
\[M^\sharp_{p,Q_0} (Tf )(x)\le c_p \sup_{x\in Q, Q\subset Q_0}\inf_{y\in Q } M_rf(y)\,,
\]
 with $c_p\to \infty$ as $p\to r$,
and, consequently,
\begin{equation}
M^\sharp_p (Tf )(x)\le c_p\, M_r f(x)\,,
\end{equation}
 with the same $c_p$ as above.

Furthermore,
\begin{equation}
M^\sharp_r  (Tf  )(x)\le c  \, M_{L^r\log L}f(x)\,.
\end{equation}
\end{theorem}

\begin{proof}
Fix a cube $Q_0$ and $x\in Q_0$. By Theorem 5.1 with $w=v=1$, $\Phi(t)=t^p$, and $0<p<r$, since $(M_rf)^p = (M(|f|^r))^{p/r}\in A_1$, for a cube
$Q\subset Q_0$ containing $x$ we have
\begin{align*} \frac1{|Q |} \int_{Q } |Tf(y) &- m_{Tf}(t,Q )|^p \,dy\\
& \le c\, \frac1{|Q |} \int_{Q  }   M_rf(y)^p \, dy
\le c_p \inf_{y\in Q } M_rf(y)^p \,.
\end{align*}
Therefore taking the supremum over $Q \subset Q_0$ containing $x$, it follows that
\[M^\sharp_{p,Q_0} (Tf )(x)\le c_p  \sup_{x\in Q, Q\subset Q_0}\inf_{y\in Q } M_rf(y)
\]
where $c_p\to \infty$ as $p\to r$.
Furthermore,  taking the supremum over all cubes $Q_0$ containing $x$ it readily follows that
\[ M^\sharp_p   (Tf   )(x) \le c_p\, M_r f (x)\]
with the same $c_p$ as before.
This gives (5.1).

Next, fix a cube $Q_0$ and $x\in Q_0$. By Theorem 5.1 with $w=v=1$ and $\Phi(t)=t^r$,
for a cube $Q\subset Q_0$ containing $x$ we have
\[ \frac1{|Q |} \int_{Q } |Tf(y) - m_{Tf}(t,Q )|^r \,dy \le c \,\frac1{|Q |} \int_{Q }   M_rf(y)^r \, dy\,.\]
Therefore, taking the supremum over $Q \subset Q_0$ containing $x$ it follows that
\[ M^\sharp_{r, Q_0}(Tf)(x)\le  c \,\sup_{x\in Q,Q\subset Q_0}\Big( \frac1{|Q |} \int_{Q }   M_rf(y)^r \, dy\Big)^{1/r}\,.
\]
Furthermore,   taking the supremum over those cubes $Q_0$ containing $x$ it readily follows that
\[ M^\sharp_r (Tf)(x)\le c\, M_r (M_r(f) )(x)\]
and   since $M_r \circ M_r$ is pointwise comparable to
the maximal operator $M_{L^r \log L}$ \cite{Compo},
(5.2) follows.
\end{proof}

Of course, it is of interest to remove the maximal function  on the right-hand side of Theorem 5.1.
The answer is   precise  for $A_p$ weights. Given a Young function $\Phi$ such that it and its
conjugate $\overline \Phi$ satisfy the $\Delta_2$ condition, recall that the upper index $u_\Phi$ of
 $L^\Phi$ is given by
 \[ u_\Phi=\lim_{s\to 0^+} - \, \frac{\ln h(s)}{\ln s}\,,\quad h(s)=\sup_{t>0}\frac{\Phi^{-1}(t)}{\Phi^{-1}(st)}\,.
\]
Then     the integral inequality
\begin{equation}
\int_{\R^n} \Phi (Mf(x) )\,w(x)\,dx\le c\int_{\R^n} \Phi(|f(x)|)\,w(x)\,dx
\end{equation}
holds if and only if $w \in A_p$ where  $p=1/u_\Phi$, \cite{KermanTorchinsky}.

Then, (1.2) can be formulated as follows.
\begin{theorem} Let $T$ and $\Phi$ be as in Theorem 5.1
with $M_T=M$.  Then, if $w \in A_p$  where  $p=1/u_\Phi$,
\[\int_{\R^n} \Phi   (|Tf(x)| )\,w(x)\,dx\le c\int_{\R^n} \Phi(|f(x) |)\,w(x)\,dx \,.\]
\end{theorem}

As for  general weights, we have, as described in (1.8),
\begin{theorem} Let $T$ be a Calder\'on-Zygmund singular integral
operator of weak-type $(1,1)$ with $M_T=M$, $\Phi,\Psi$ Young functions such that
\[\int_0^t \frac{\Phi(s)}{s^2}\, {ds} \le c \,\frac{\Psi(t)}{t}\,,\quad t>0\,, \]
 and  $\Phi(t)/t^q$ decreases for some $ 1< q<\infty$, and    $(w,v)$ satisfy condition $F$.
Then
\begin{equation} \int_{\R^n} \Phi (|Tf(x)|  )\,w(x)\,dx \le c\int_{\R^n} \Psi(|f(x)|)\,Mv(x)\,dx\,.
\end{equation}

Moreover, if $w\in C_{q}$ for some  $1<q<\infty$,
\begin{equation} \int_{\R^n} \Phi  (|Tf(x)| )\, w( x )\, dx\le c\, \int_{\R^n}\Psi(|f(x)|) M_{L\log L}(w)( x )\, dx
 \,.
\end{equation}
\end{theorem}

\begin{proof}
Fefferman and Stein observed that for a weight $u$,   $M$ is bounded from $L^\infty(Mu)$ to $L^\infty(u)$ and
  it  maps $L^1(Mu)$ weakly into $L^1(u)$, \cite{FeffermanStein}. Then for $\Phi,\Psi$ as above
a simple interpolation argument \cite{Stromberg, TorInt} gives that
\begin{equation}
\int_{\R^n} \Phi (Mf(x) )\,u(x)\,dx\le c\int_{\R^n} \Psi(|f(x)|)\,Mu(x)\,dx\,.
\end{equation}

Therefore,  if $T$ satisfies the conditions of Theorem 4.2 and $(w,v)$ satisfy condition $F$,
\begin{align} \int_{\R^n} \Phi    (|Tf(x)|  )\,w(x)\,dx &\le \int_{\R^n} \Phi  (Mf(x)  )\,v(x)\,dx
\notag
\\
&\le c\int_{\R^n} \Psi(|f(x)|)\,Mv(x)\,dx\,.
\notag
\end{align}

Moreover, if $w\in C_{q}$ for some  $1<q<\infty$,   since $(w, Mw)$ satisfy condition $F$
and $ M \circ M \sim  M_{L\log L}$, by (5.6),
\[ \int_{\R^n} \Phi   (|Tf(x)|  )\, w( x )\, dx\le c\, \int_{\R^n}\Psi(|f(x)|) M_{L\log L}(w)( x )\, dx
 \,.\]
 This completes the proof.
\end{proof}

In the case of $A_\infty$ weights the result is of interest when
$p=1/u_\Phi$ and $w\in A_q$, with $1<p<q<\infty$. Similarly,
if $p=1/u_\Phi$, the result is of interest when  $w\in C_{q}$ where $q> p$.

Note that Theorem 5.4 implies that for every integer  $k$ there exists a weight $w$ such that  $(w,M_{A_k}w)$
do not satisfy condition $F$, where the $A_k$ are as in Section 3. For the sake of argument suppose that
 $(w,M_{A_k}w)$ satisfy condition $F$ for all weights $w$. Then by (5.4),
\begin{align*}
\int_{\R^n} |Tf(x)|^q\, w(x)\,dx &\le c\, \int_{\R^n} |f(x)|^q\,M( M_{A_k}w)(x)\,dx
\\
&\le c\, \int_{\R^n} |f(x)|^q\,  M_{A_{k+1}}w(x)\,dx
\end{align*}
for all singular integrals $T$ and all $q$, which contradicts (1.11) for sufficiently large $q$.

And, if $M_T=M_r$, $1\le r<\infty$, we have
\begin{theorem}
Let $T$ be an operator that satisfies the conditions of Theorem 4.2 with $M_T=M_r$, $1\le r<\infty$.
Let $\Phi$ satisfy condition $C$ be such that $\Phi(t)/t^p$ increases and $\Phi(t)/t^q$
decreases for some $r<p<q<\infty$. Then if $\Psi(t)=\Phi(t^{1/r})$ is convex, $p=1/u_\Phi$,
and $w\in A_{p/r}$,
\[ \int_{\R^n}   \Phi   ( |Tf(x)| )\,w(x)\, dx \le c\, \int_{\R^n}   \Phi( |f(x)|  ) \,w(x)\, dx\,.
\]
\end{theorem}
\begin{proof}
By Theorem 5.1,
  for any $\Phi$ satisfying condition $C$ and $w\in A_\infty$,
\[\int_{\R^n} \Phi  (|Tf(x)| ) \,w(x)\,dx \le c \int_{\R^n}   \Phi  ( M_r f(x)  )\, w(x)\, dx\,.\]

Now, since  $\Psi^{-1}(t) =\Phi^{-1}(t)^r  $,
by a simple computation $u_\Psi=r u_\Phi =r/p$. Thus, if $w\in A_{p/r}$, by (5.3),
\begin{align*}
\int_{\R^n}   \Phi   ( M_r f(x)  ) \,w(x)\, dx &=\int_{\R^n}   \Psi  ( M (|f|^r  ) (x)  )\, w(x)\, dx\\
&\le c\, \int_{\R^n}   \Psi( |f(x)|^r  ) \,w(x)\, dx\\
&= c\,\int_{\R^n}   \Phi( |f(x)|  )\, w(x)\, dx\,,
\end{align*}
and we have finished.
\end{proof}
That this result is essentially sharp when $\Phi$ is a power is discussed in \cite{KurtzW}.

Now, observe that (5.5)  is reminiscent of the  estimate
\begin{equation*}
\int_{\R^n} |Tf(y)|^p\, w(y)\,dy \le c \int_{\R^n} |f(y)|^p M_{{A}} w(y)\,dy
\end{equation*}
established by P\'erez,
which   holds for any weight $w$  if $A\in B_p$, i.e.,
the doubling Young function $A$ is such that
\begin{equation*}
\int_c^{\infty} \frac{A(t)}{t^p} \frac{dt}{t} < \infty
\end{equation*}
for some $c > 0$, \cite{Perez1990}.

Since   $B_p$ implies   $B_q$ for $p<q<\infty$,  if $A$ satisfies condition $B_p$,
  $T$ maps continuously $L^p(M_Aw)$ into $L^p(w)$ and
  $L^q(M_A w)$ into $L^q(w)$ for every $q>p$. Then, if $\Phi$ is such that
  $\Phi(t)/t^p$ increases and $\Phi(t)/t^q$ decreases  for some $q>p$, by essentially the same
  interpolation argument as before
\begin{equation*}
\int_{\R^n} \Phi (|Tf(y)| )\, w(y)\,dy \le c \int_{\R^n} \Phi(|f(y)|)\, M_{{A}} w(y)\,dy\,.
\end{equation*}

We now consider  the case $M_T=M_A$.
\begin{theorem} Let $T$ be a singular integral operator as in Theorem 5.1
with $M_T=M_A$,
 weights $(w,v)$ that satisfy property $F$, and  $A\in B_p$. Then, if
$\Phi(t)/t^p$ increases and $\Phi(t)/t^q$ decreases  for some $q>p$,
we have
\begin{equation*} \int_{\R^n} \Phi  (|Tf(y)|)\, w(y)\,dy  \le c \int_{\R^n} \Phi(|f(y)|)\, Mv(y)\,dy\,.
\end{equation*}
\end{theorem}
\begin{proof}
First, by Theorem 5.1,
\begin{equation}
\int_{\R^n}\Phi   (|Tf(x)|  )\,w(x)\, dx\le c\int_{\R^n}\Phi  (M_Af(x)  )\,v(x)\, dx\,.
\end{equation}

Next, recall that by Theorem 1.7 in \cite{Perez1990}, $A\in B_p$ if and only if for all weights $v$,
\[\int_{\R^n}M_Af(x)^p\,v(x)\, dx \leq c\int_{\R^n}|f(x)|^p\,Mv(x)\, dx\,,
\]

Now, $M_A$ maps continuously $L^p(Mv)$ into $L^p(v)$ and
  $L^q(Mv)$ into $L^q(v)$ for every $q>p$. Then if $\Phi$ is such that
  $\Phi(t)/t^p$ increases and $\Phi(t)/t^q$ decreases  for some $q>p$,
  interpolating we have
\begin{equation}
\int_{\R^n} \Phi (M_Af(y) )\, v(y)\,dy \le c \int_{\R^n} \Phi(|f(y)|)\, Mv(y)\,dy\,.
\end{equation}

The conclusion follows combining  (5.7) and (5.8).
Note that, in particular, if $w\in A_\infty$, $v=w$.
\end{proof}

Lastly, we discuss the integral operators with homogeneous  kernels.

\begin{theorem}
Let $T$ be an integral operator with  homogeneous kernel as in {\rm (4.14)},
$\Phi$ a Young function, $p=1/u_\Phi$, and  $w \in A_p$, $1 < p < \infty$,
such that  $w(A_i x)\le c\, w(x)$ for a.e.\! $x\in  \R^n$, all $i$.
Then, if  $ \ \lim_{Q_0 \to \R^n}m_{Tf}(t,Q_0) = 0$,
\[\int_{\R^n} \Phi   (|Tf(x)|  )\, w(x)\, dx \le c   \int_{\R^n} \Phi  (  | f(  x)|  )\, w( x)\,dx\,.\]
\end{theorem}

\begin{proof}
By Theorem 5.1 and \cite{KermanTorchinsky},
\begin{align*}
\int_{\R^n} \Phi  (|Tf(x)|  )\, w(x)\,dx
&\le c \sum_{i=1}^m   \int_{\R^n} \Phi  ( Mf(A_i^{-1}x) )\, w(x)\,dx \\
&\le c \sum_{i=1}^m   \int_{\R^n} \Phi  ( Mf(x)  )\, w( A_i x)\,dx \\
&\le c  \sum_{i=1}^m   \int_{\R^n} \Phi  ( Mf(x)  )\, w(   x)\,dx\\
&\le c   \int_{\R^n} \Phi  (  | f(  x)|  )\, w( x)\,dx\,,
\end{align*}
and we have finished.
\end{proof}

A   simple computation shows that if $w(x)\in A_p$, then $w(\lambda x)$ is in $A_p$ with the same constant for
$\lambda>0$. So, when the $A_i$ are diagonal matrices with diagonal element $a_i>0$, $1\le i\le m$, as in \cite{RiverosUrciuolo},
without any additional assumpations on $w$,  by \cite{KermanTorchinsky} the conclusion is
\begin{equation*} \int_{\R^n} \Phi  (|Tf(x)|  )\, w(x)\,dx
\le c\, \sum_{i=1}^m   \int_{\R^n} \Phi( | f( A_i^{-1} x)| )\, w(  x)\,dx\,.
\end{equation*}

\section{$L^p_v-L^q_w$ estimates for singular integral operators, $1<p\le q<\infty$}

In this section we consider two-weight $L^p$ estimates for a fixed $1<p<\infty$
and two-weight, $L^p_v-L^q_w$ estimates with $1<p\le q<\infty$  that apply directly to a
 Calder\'{o}n-Zygmund singular integral operator 
 and where the control exerted by a maximal function is not  apparent.
 The  strategy to deal with these operators
follows the ideas developed so far: we consider  local median decompositions and
  weights naturally
 associated to them.  The first scenario corresponds to
 the $W_p$ classes of Fujii \cite{Fujii1991}.

\begin{mydef}   Fix $1 < p < \infty$. We say that the weights $(w, v)$ satisfy condition $W_p$
if there exist positive constants $\alpha, \beta, c_0$, with $\alpha <1$, so that for every cube $Q$ and for all measurable
$E, E' \subset Q$ with $E \cap E' = \emptyset$ and $|E'| \ge \alpha |Q|$,
\[\int_E w(x)\,dx \, \Big ( \frac{1}{|Q|} \int_{c(n,\alpha )Q} v(x)^{1-p'}\, dx \Big )^p
\le c_0 \Big ( \frac{|E|}{|Q|} \Big )^{\beta} \int_{E'} v(x)^{1-p'}\,dx < \infty\]
where $c(n,\alpha) > 1$   is increasing with respect to $\alpha$. For such weights we write $(w,v) \in W_p\,$.
\end{mydef}

Fujii notes that for $w = v$, $W_p$ is equivalent to the $A_p$ condition.
He also shows that $W_p$ implies Sawyer's testing condition for the two-weight, $(p,p)$ boundedness of $M$, \cite{Sawyer1982}.
Also, note that if $(w,v)\in W_p$,   $1 < p < \infty$,
$(w,v)$ satisfy condition $ F$.
Indeed,  observe that for a fixed $0<\alpha <1$,
for all $E\subset Q$ with $|E|\le (1-\alpha)|Q|$, by H\"older's inequality
\[ 1 \le c\, \Big(\frac{\int_{Q \setminus E} v(x)\,dx}{\int_{Q \setminus E} v(x)^{1-p'}\,dx}\Big)\Big(\frac{1}{|Q|}\int_{c(n,\alpha) Q} v(x)^{1-p'}\,dx\Big)^{p} \,,
\]
and therefore
\begin{align*}
\int_E w(x)\,dx &
\le  c\,c_0 \Big ( \frac{|E|}{|Q|} \Big )^{\beta} \Big(\int_{Q \setminus E} v(x)^{1-p'}\,dx \Big)
\Big(\frac{\int_{Q \setminus E} v(x)\,dx}{\int_{Q \setminus E} v(x)^{1-p'}\,dx}\Big)
\\
&\le  c\, \Big ( \frac{|E|}{|Q|} \Big )^{\beta}  \int_{Q \setminus E} v(x)\,dx\,,
\end{align*}
which gives condition $F$.

The variant of the decomposition of Section 2 that corresponds to these weights is sketched below and
is referred to as the annular decomposition.
\begin{mydef}
For $0 < s \le 1/2$ and a measurable function $f$, we define
\[m^{\sharp}_{f}(1-s,Q) = \inf_c m_{|f - c|}(1-s,Q).\]
\end{mydef}
Recall that by (4.3) of \cite{MedContOsc} we have that
\begin{equation*}
m^{\sharp}_f(1-s,Q) \le  m_{|f - m_f(1-s,Q)|}(1-s,Q)\le 2\,m^{\sharp}_f(1-s,Q)\,.
\end{equation*}

The annular decomposition is obtained by establishing
the analogues to Lemma 4.1 and Lemma 4.3 in \cite{MedContOsc} in this context, and
then the decomposition   follows entirely as Theorem 2.1, but with different bounds on the constants $a^v_j$.

\begin{theorem}
Let $f$ be a measurable function on a fixed cube $Q_0 \subset \mathbb{R}^n$, $0 < s < 1/2$, and $1/2 \le t < 1-s$. Then {\rm (ii)-(iv)} of Theorem 2.1 hold, and for a.e. $  x \in Q_0$,
\[|f(x) - m_f(t,Q_0)| \le  4M^{\sharp}_{0,s}f(x) + \sum_{v=1}^{\infty}\sum_{j \in I^v_2} a^v_j \indicator{Q^v_j}(x) \,,\]
where
\begin{equation*}
a^v_j \le (10n+2)\sup_{ Q_0\supset Q \supset Q^v_j }m^{\sharp}_f(1-s,Q)\,.
\end{equation*}
\end{theorem}

Then, there is the corresponding pointwise inequality similar to Theorem 4.1.
\begin{theorem}
Let $T$ be a singular integral operator satisfying the conditions of Theorem 4.1 with $r=1$.
Then for  $0 < s \le  1/2$ and any cubes $Q_1 \supset Q_0$,
\begin{equation}
\sup_{Q_1 \supset Q \supset Q_0 } m^{\sharp}_{Tf}(1-s,Q) \le c \sup_{Q_1 \supset Q \supset Q_0 } \frac{1}{|Q|}\int_Q |f(y)|\,dy\,.
\end{equation}

Moreover, if we also have that $T(1)=0$,
then
\begin{equation*}
\sup_{Q_1 \supset Q \supset Q_0} m^{\sharp}_{Tf}(1-s,Q) \le c \sup_{Q_1 \supset Q \supset Q_0 }\frac{1}{|Q|}\int_Q |f(y) - f_Q|\,dy\,.
\end{equation*}
\end{theorem}

The proof of this result is based on what are by now familiar ideas and is therefore omitted.

We can now prove the two-weight $L^p$ boundedness.

\begin{theorem}
Let $1 < p < \infty$. Suppose that $(w,v) \in W_p$.
Let $T$ be a singular integral operator satisfying the conditions of Theorem 4.1
with $r=1$.
Then if  $f$ has support  contained in a cube $Q_0$,
\begin{equation*}
\int_{Q_0}  | Tf(x) - m_{Tf}(t,Q_0)  |^p \, w(x)\,dx \le c \int_{Q_0}|f(x)|^p \,v(x)\,dx\,.
\end{equation*}

Furthermore, if  $m_{Tf}(t,Q_0) \to 0$ as $Q_0 \to \mathbb{R}^n$, then
\begin{equation}
\int_{\mathbb{R}^n} |Tf(x)|^p \,w(x)\,dx \le c\int_{\mathbb{R}^n} |f(x)|^p\, v(x)\,dx\,.
\end{equation}
\end{theorem}

\begin{proof}
We only sketch the proof, which  follows the blueprint of the proof of Theorem 3.1. By Theorem 6.1, for a.e. $\!x \in Q_0$,
\[|Tf(x) - m_{Tf}(t,Q_0)|  \le  4 M^{\sharp}_{0,s,Q_0}  (Tf  )(x) + \sum_{v=1}^{\infty} \sum_{j \in I^v_2} a^v_j\indicator{Q^v_j}(x),\]
where
\[a^v_j \le  (10n+2)\sup_{Q_0\supset Q\supset Q^v_j }m^{\sharp}_{Tf}(1-s,Q)\,.\]

Then
\[  |Tf(x) - m_{Tf}(t,Q_0)  |^p \le c_p\,   ( a(x,Q_0,p) + b(x,Q_0,p) )\,,\]
where  we define
\begin{equation*}
a(x,Q_0,p)  =   M^{\sharp}_{0,s}(Tf)(x)^p + \sum_{k=1}^{\infty} \sum_{j \in I^k_2}   |a^k_j|^p\indicator{Q^k_j \setminus \Omega^{k+1}}(x)
\end{equation*}
and
\begin{equation*}
b(x,Q_0,p)  = \sum_{k=2}^{\infty} \indicator{\Omega^k \setminus \Omega^{k+1}} \sum_{v=1}^{k-1} c_p^{k-v} \sum_{j \in I^v_2} |a^v_j |^p\indicator{Q^v_j}(x)\,.
\end{equation*}

Thus it suffices to prove that
\[\int_{Q_0} a(x,Q_0,p)\, w(x)\,dx \le  c \int_{Q_0} |f(x)|^p v(x)\,dx\]
and a similar estimate with $ b(x,Q_0,p)$ in place of $ a(x,Q_0,p)$ above.
By (4.3) and (6.1) we have that $a(x,Q_0,p)\le c\, Mf(x)^p$, and since ${\rm supp}(f)$ $\subset Q_0$,
that $W_p$ implies Sawyer's testing condition allows us to handle the first inequality. As for the second,
    we use (6.1) and that $(w,v) \in W_p$ as in the proof of the
    Theorem   in \cite{Fujii1991}. That (6.2) holds follows immediately from Fatou's lemma.
  \end{proof}

Condition $W_p$ also gives continuity results for singular integral operators for values of $q\ne p$.
Indeed,  we have
\begin{theorem}
Let $1 < p < \infty$. Suppose that $(w,v) \in W_p$. Let $T$ be a singular integral operator satisfying the conditions of Theorem 4.1
with $r=1$.

Then,  if $1<q<p$, and $0<\eta<p'/q'$,
 \begin{equation}
 \int_{\R^n} |Tf(x)|^q\, w(x)\,dx\le c \int_{\R^n} |f(x)|^q\,  (v(x)/Mv(x))^\eta\, Mv(x)\,   dx\,.
 \end{equation}

And, if $p<q<\infty$ and $0< 1-\eta<p/q<1$,
\begin{equation}
\int_{\R^n} |Tf(x)|^q\, w(x)\,dx\le c \int_{\R^n} |f(x)|^q\, ( Mv(x)/v(x))^\eta \,v(x)\,dx\,.
\end{equation}
\end{theorem}

\begin{proof}  Since $(w,v)$ satisfy
property $F$,  by Theorem 5.4,
 \[\int_{\R^n} |Tf(x)|^r\, w(x)\,dx\le  c_r \int_{\R^n} |f(x)|^r\,Mv(x)\,dx\,,\]
for all $1<r<\infty$. Now, if $1<q<p$ and   $0<\eta<p'/q'$, the index $r$ defined  by
the relation
 \begin{equation*}
 \frac1{r}= \frac1{1-\eta}\,\Big(\frac1{q}-\frac{\eta}{p}\Big)
\end{equation*}
satisfies  $1<r<q<p$ and  $1/q=(1-\eta)/r+\eta/p$, and
since $T$ maps
$L^r(Mv)$ into $L^r(w)$ and
 $L^p(v)$ into $L^p(w)$, the conclusion follows
by the Stein-Weiss theorem of interpolation with change of measure.

Now, in  case  $p<q$ and $1-\eta<p/q<1$,  the index $r$ given by
\[
\frac1{r}= \frac1{\eta}\,\Big(\frac1{q}-\frac{1-\eta}{p}\Big)
\]
satisfies  $p<q<r<\infty$ and  $1/r=(1-\eta)/q+\eta/p$, and
since $T$ maps $L^p(v)$ into $L^p(w)$ and $L^r(Mv)$ into $L^r(w)$, the conclusion also follows
by the Stein-Weiss theorem of interpolation with change of measure.
\end{proof}

While (6.4) is reminiscent of extrapolation estimates \cite{CUMP},
the estimate (6.3) for values $q<p$  is not.

So far, we have produced two median  function decompositions
leading to two-weight  continuity results for Calder\'on-Zygmund
singular integral operators.
 Now, each decomposition generates families of cubes that share certain properties, and it is these properties and not the specific cubes that is of interest. In particular, the bounds on the respective  $a^v_j$ are related, and in this way the annular decomposition is stronger than the first decomposition. Indeed, it is readily seen that for a cube $Q^v_j$ generated by the annular decomposition,
\[\sup_{Q_0 \supset Q \supset Q^v_j} m^{\sharp}_f(1-s,Q) \le \inf_{y \in Q^v_j} M^{\sharp}_{0,s,Q_0}f(y)\,.\]
(This right-hand side is the bound for the $a^v_j$ used when invoking the first decomposition, as in the proof of Theorem 3.1.)

If we are more deliberate in the construction of the
local median decomposition, an even sharper bound for the $a^v_j$ results. This third
decomposition is needed for other applications, including Lerner's proof of the $A_2$ conjecture \cite{La}.

\begin{theorem}
Let $f$ be a measurable function on a fixed cube $Q_0 \subset \mathbb{R}^n$, $0 < s < 1/2$, and $1/2 \le  t < 1-s$. Then {\rm (ii)-(iv)} of Theorem 2.1 hold, and for a.e. $  x \in Q_0$,
\[|f(x) - m_f(t,Q_0)| \le  8M^{\sharp}_{0,s}f(x) + \sum_{v=1}^{\infty}\sum_{j \in I^v_2} a^v_j \indicator{Q^v_j}(x) \,,\]
where
\[ a^v_j \leq m_{|f - m_f(t,\widehat{Q^v_j})|}  (1-(1-t)/2^n,\widehat{Q^v_j} ).\]
\end{theorem}

\begin{proof}
We follow the proof of Theorem 2.1 in form, with
a few definitional changes. First note that for any cube $Q$,
\begin{equation}
m_{|f - m_f(t,Q)|}(t,Q) \leq  4\inf_{y \in Q}M^{\sharp}_{0,s,Q}f(y)\,.
\end{equation}
To see this, from Proposition 1.1 and (4.3) in \cite{MedContOsc},
\begin{align*}
m_{|f - m_f(t,Q)|}(t,Q) &\le m_{|f - m_f(1-s,Q)| + |m_f(1-s,Q) - m_f(t,Q)|}(t,Q)
\\
&\le 2\, m_{|f - m_f(1-s,Q)|}(t,Q) \le 2\,m_{|f - m_f(1-s,Q)|}(1-s,Q)
\\
&\le 4\inf_{y \in Q}M^{\sharp}_{0,s,Q}f(y)\,.
\end{align*}

We define $E^1 = \{x \in Q_0: |f(x) - m_f(t,Q_0)| > m_{|f - m_f(t,Q_0)|}(t,Q_0)\}$. If $|E^1| = 0$, the decomposition halts, just as in Theorem 2.1. So we suppose $|E^1| > 0$. We then define
\[\Omega^1 = \big \{x \in Q_0: m^{t,\Delta}_{Q_0}(f^0)(x) > m_{|f - m_f(t,Q_0)|}(t,Q_0) \big \}.\]

Proceeding as before, we have that $\Omega^1 = \bigcup_j Q^1_j$ so that (as in (2.3))
\begin{equation}
|m_{f^0}(t,Q^1_j)| > m_{|f^0|}(t,Q_0), \quad {\text{and}}\quad |m_f(t, \widehat{Q^1_j})| \leq m_{|f^0|}(t,Q_0)\,.
\end{equation}
Furthermore, we also have that
\begin{equation*}
\sum_j |Q^1_j| \le \frac{s}{1-t}\,|Q_0|\,.
\end{equation*}

Before continuing, observe that for any cube $Q$,
\begin{equation}
m_f(t,Q) \le  m_f  (1-(1-t)/2^n,\widehat{Q} )\,.
\end{equation}
To see this, note that
\begin{align*}
|\{y \in \widehat{Q}: f(y) \ge  m_f(t,Q)\}| &\ge  |\{y \in Q: f(y) \ge  m_f(t,Q)\}|
\\
&\ge (1-t)|Q| = \frac{1-t}{2^n}| \widehat{Q} |\,,
\end{align*}
so taking complements in $\widehat{Q}$ we have
\[|\{y \in \widehat{Q}: f(y) < m_f(t,Q)\}| \le  \Big ( 1 - \frac{1-t}{2^n} \Big ) |\widehat{Q}|\,.\]
Note also that by our choice of $t$, it follows that $1 - (1-t)/2^n \ge  1/2$.

Let $\alpha^1_j = m_{f^0}(t,Q^1_j)$. By (6.6) and (6.7) we have
\begin{align*}
|\alpha^1_j| &\le  |m_{f^0}(t,Q^1_j) - m_{f^0}(t, \widehat{Q^1_j})| + |m_{f^0}(t,\widehat{Q^1_j})|
\notag
\\
&\le  m_{|f^0 - m_{f^0}(t,\widehat{Q^1_j})|}(t,Q^1_j) + m_{|f^0|}(t,Q_0)
\notag
\\
&\le  m_{|f - m_f(t,\widehat{Q^1_j})|}\Big (1-(1-t)/2^n,\widehat{Q^1_j} \Big) + m_{|f^0|}(t,Q_0)\,.
\end{align*}

This gives the first iteration of the
decomposition of $f$ when $|E^1|$ $ > 0$: for a.e. $  x \in Q_0$, with $g^1 = f^0 \indicator{Q_0 \setminus \Omega^1}$,
\[f^0(x) = g^1(x) + \sum_j \alpha^1_j \indicator{Q^1_j}(x) + \sum_j   ( f^0(x) - m_{f^0}(t,Q^1_j) )\indicator{Q^1_j}(x).\]
Clearly by (6.5)
\[|g^1(x)| \le  m_{|f - m_f(t,Q_0)|}(t,Q_0) \le  4\inf_{y \in Q_0}M^{\sharp}_{0,s,Q_0}f(y) \le  4M^{\sharp}_{0,s,Q_0}f(x)\]
a.e.  on $Q_0 \setminus \Omega^1$.

By proceeding as in the proof of Theorem 2.1, the result follows.
\end{proof}

This decomposition generates families of cubes sharing the same properties as those
from the first and annular decompositions.
And, as anticipated, for some parameters $s, t$ the bound on the $a^v_j$ from this decomposition is even smaller.
Indeed, for any $c$, since $1/2 \le  t \le 1-(1-t)/2^n$,
\begin{align*}
m_{|f - m_f(t,\widehat{Q^v_j})|}&(1-(1-t)/2^n,\widehat{Q^v_j})
\\
&\le m_{|f - c|}(1-(1-t)/2^n,\widehat{Q^v_j}) +  |c -m_f(t,\widehat{Q^v_j})  |
\\
&\le m_{|f - c|}(1-(1-t)/2^n,\widehat{Q^v_j}) + m_{|f-c|}(1-(1-t)/2^n,\widehat{Q^v_j})
\\
&\le 2 m_{|f - c|}(1-(1-t)/2^n,\widehat{Q^v_j})\,.
\end{align*}
Thus
\[m_{|f - m_f(t,\widehat{Q^v_j})|}(1-(1-t)/2^n,\widehat{Q^v_j}) \le 2\, m^{\sharp}_f (1-(1-t)/2^n,\widehat{Q^v_j})\,.\]
Then for any $0 < s < 1/2^{n+1}$ and $1/2 \le t \le 1-2^n s$,
\[m_{|f - m_f(t,\widehat{Q^v_j})|}(1-(1-t)/2^n,\widehat{Q^v_j}) \le c \sup_{Q_0 \supset Q \supset Q^v_j}m^{\sharp}_f(1-s,Q)\,.\]

Before we proceed to prove our next theorem, we need a
couple of preliminary results. The first is
an extension of a property given in Lemma 4.8 in \cite{JawerthTorchinsky} and the comments that follow it.
\begin{lem}
Let $T$ be a Calder\'on-Zygmund singular integral operator
defined by { \rm (4.2)} and $Q$ a cube of $\R^n$.
If $T$ satisfies the assumptions of Theorem 4.1 with $1\le r<\infty$, let
 \[\lambda_m=  \omega (c_n/{2^m})\,,\quad m\ge 1\,.\]
If $T$ satisfies the assumptions  of Theorem 4.3 for the Young function $A(t)=t^{r'}$, $1\le r<\infty$, let
 \[ \lambda_m = \sup_{u,v\in Q}
|2^{m+1}Q| \,  \| \indicator{2^{m+1}Q\setminus  2^{m}Q} ( k (u, \cdot)- k(v,\cdot)   )\|_{L^{{r'}}(2^{m+1}Q)} \,.
\]
In either case, we have $\sum_{m=1}^\infty \lambda_m<\infty$ and
\begin{equation}
m^\sharp_{Tf}(1-s, Q) \le
  c\,  \sum_{m=1}^{\infty}  \lambda_m  \,\Big(
 \frac{1}{|2^{m}Q|} \int_{2^m Q} |f(y) |^r \, dy \Big)^{1/r}.
 \end{equation}
\end{lem}

\begin{proof}
Fix $Q$,  let $x_Q\in Q$, and put $f=f_1+f_2$ where $f_1=f\indicator{2Q} $.
We claim that there exist constants $c_1,c_2>0$ independent of $f$ and $Q$ such that
\begin{equation}
|\{z \in Q: |Tf_1(z)| >  c_1\, I \}| < s\,|Q|\,,
\end{equation}
and
\begin{equation}
  \|Tf_2 - Tf_2(x_Q) \|_{L^{\infty}(Q)} \le  c_{2} \, I\,,
\end{equation}
where
\[ I= \sum_{m=1}^{\infty} \lambda_m  \,\Big(\frac{1}{|2^{m}Q|} \int_{2^m Q} |f(y) |^r \, dy \Big)^{1/r}.
\]

Now, if  $T$ satisfies the assumptions
of Theorem 4.1, (6.9) follows since $T$ is of weak-type $(r,r)$, and
(6.10) follows as in the proof of (4.10) followed by H\"older's inequality when $r>1$.

And, when $T$ satisfies the assumptions
of Theorem 4.3, (6.10) holds automatically.  (6.9) follows using
that $T$ is of weak-type $(1,1)$ and H\"older's inequality when $r>1$.

 Then, in either case, with $c > \max \{c_1, c_2\}$, (6.9) and (6.10) give
\begin{equation*}
|\{z \in Q : |Tf(z) - Tf_2(x_Q)| > 2c \, I\}| <{s}\,|Q|\,.
\end{equation*}
Whence for all $Q$,
\begin{align*} m^\sharp_{Tf}(1-s, Q) &=\inf_{c'} \; \inf \{\alpha \ge  0: |\{z \in Q: |Tf(z) - {c'}| > \alpha\}| < s|Q|  \}\\
&\le c\, \sum_{m=1}^{\infty} \lambda_m  \,\Big( \frac{1}{|2^{m}Q|} \int_{2^m Q} |f(y) |^r\, dy\Big)^{1/r}
\end{align*}
and (6.8) holds.
\end{proof}

Note that Lemma 6.1 also applies to the Calder\'on-Zygmund singular integral operators of Dini type.
In that case, using Lemma 5 in \cite{KurtzW}, the $\lambda_m$ can be estimated in terms
of the $\omega_{r'}$ modulus of continuity of   $\Omega$.

Next, we collect some properties of Young functions in the classes $B_p$ and $B_p^\alpha$. The latter class was introduced by
Cruz-Uribe and Moen and for  $0<\alpha<1$ and  $1<p<1/\alpha$, it consists of those Young functions $A$ such that with $1/q= 1/p-\alpha$,
\begin{equation}
\|A\|_{\alpha,p}=\Big(\int_c^\infty \frac{A(t)^{q/p}}{t^q}\,\frac{dt}{t}\Big)^{1/q}<\infty\,.
\end{equation}

As they point out, if $\alpha > 0$,  $B_{p}^\alpha$
 is weaker than $B_p$. The  result of interest to us, Theorem 3.3 in \cite{CUM}, is that
 for   $A\in B_p^\alpha$    the maximal function
\[M_{\alpha, A}f(x)= \sup_{x\in Q} |Q|^\alpha\|f\|_{L^A(Q)}
\]
maps $L^p(\R^n)$  continuously into  $L^q(\R^n)$ with norm
$ \le c\,\|A\|_{\alpha,p}\,$.

We also have,

\begin{prop}
Let $A$ be a Young function and $C(t)= A(t^r)$ for $1<r<\infty$.
\begin{enumerate}
\item[\rm (i)]
If $A\in B_p$, then $C\in B_{rp}$.
Furthermore,   for all cubes $Q$,
\begin{equation}
 \| g \|_{L^{C}(Q)}=  \| g^{r} \|_{L^{A}(Q)}^{1/r} \,.
 \end{equation}
  \item[\rm(ii)]
If $\overline{A} \in B_{p\,'}$,
there exists a positive constant $c$
independent of  $g$ and $Q$ such that
\begin{equation}
 \| g\|_{L^p(Q)}\le c\, \|g\|_{L^A(Q)}\,.
 \end{equation}
\item[\rm (iii)]
  If     $A\in B^{\alpha r}_{p/r}$  for $r<p$,   then  $C\in B^\alpha_p$.
\item[\rm(iv)]
If $\overline A\in B_p^\alpha$, then $\overline C\in B_p^\alpha$.
In particular, if  $\overline A\in B_p \,$, then $\overline C\in B_p$.
\end{enumerate}
\end{prop}

\begin{proof}

The proof of (i) is a straightforward computation and is therefore omitted. As for (ii),
recall that if $\overline A\in B_{p'}$,  for some constant $c>0$,
\[ \int^\infty_c \Big(\frac{t^{p}}{A(t)}\Big)^{p'-1}\,\frac{dt}{t}<\infty\,,\]
and therefore,  there exist positive constants $c_0,c_1$ such that
$A(t)\ge c_0 \,t^p$ for  $t\ge c_1$. Then,
by a direct computation or the closed graph theorem, there exists a positive constant $c$
independent of  $g$ and $Q$ such that (6.13) holds.

Now, for (iii),
let $q$ be given by the relation  $1/q=1/p-\alpha$; then
 $r/q=r/p-\alpha\,r$ and the value of $q$ in (6.11) for membership in the class
 $B^{\alpha r}_{p/r}$ is $q/r$. Then, since  $(q/r)/(p/r)=q/p$, we have
\[ \int_{c }^\infty \frac {C(t)^{q/p}}{t^{q}}\,\frac{dt}{t}=
\int_{c_1}^\infty \frac {A(t)^{(q/r)/(p/r)}}{t^{q/r}}\,\frac{dt}{t}
<\infty\,,
\]
and $C \in B_p^\alpha$.

Finally, (iv); since the proof   for $\alpha=0$ follows by setting $p=q$ in the proof
for the case $\alpha>0$, we do the latter.  Taking inverses, $ C^{\  -1}(t) = A^{\  -1} (t)^{1/r}$,
and therefore it readily follows that
$\overline C^{\  -1}(t)\sim   t^{1/r'} \, \overline{A}^{\  -1} (t)^{1/r}$.
Then,
\begin{align*}
\int_c^\infty \frac{\overline C(t)^{q/p}}{t^q} \frac{dt}{t} &\le c_1 \int_{c_2}^\infty
\frac{t^{q/p-1}}{\overline C^{\, -1}(t)^q}\, dt
\\
&\sim c_3 \int_{c_2}^\infty \frac{t^{q/p-1}}{\overline A^{\, -1}(t)^{q/r} t^{q/r'}}\,dt\\
&\sim c_4
\int_{c_5}^\infty \frac1{\overline A(t)^{q/r'}}\overline A(t)^{q/p} \frac1{t^{q/r}}\,\frac{dt}{t}\,,
\end{align*}
which, since  $\overline A(t)/t$ increases,    is bounded by
 \[c_6 \int_{c_5}^\infty \frac{\overline A(t)^{q/p}}{t^{q(1/r +1/r')}}\, \frac{dt}{t}
= c_6 \int_{c_5}^\infty \frac{\overline A(t)^{q/p}}{t^{q}}\, \frac{dt}{t}<\infty\,.
\]

This completes the proof.
\end{proof}

We will also rely on the following result of P\'erez,  Theorem 2.11 in \cite{PerezIUMJ} or
Theorem 3.5 in \cite{GCM}.
 Let $p, q$ with $1 < p \le q < \infty$, and $(w, v)$  a pair
of weights such that for every cube $Q$,
\begin{equation}
 |Q|^{1/q-1/p} \| w^{1/q}\|_{L^q(Q)}\, \| v^{-1/p}\|_{L^B(Q)} \le  c
 \end{equation}
where $B$ is a Young function with $\overline B\in B_p$.
Then, the Hardy-Littlewood maximal function $M$ maps $L^p_v(\R^n)$ continuously in $L^q_w(\R^n)$,  i.e.,
\begin{equation}
\| Mf\|_{L^q_w}\le c\,\|f\|_{L^p_v}\,.
\end{equation}

We are now ready to prove our result.

\begin{theorem} Let $T$ be a Calder\'{o}n-Zygmund singular integral operator
that satisfies the assumptions of Theorem 4.1  with $1\le r<\infty$ or  Theorem 4.3
with   the Young function $t^{r'}$  there   and $1\le r<\infty$.
Let $  r < p \le q<\infty$,   define $0\le \alpha<1$ by the relation $\alpha =1/p-1/q$,
and let $\alpha_1,\alpha_2\ge 0$ be such that $\alpha=\alpha_1+\alpha_2$.
Further, suppose that the Young functions $A$, $B$
are so that $\overline A \in  B_{(q/r)'}\cap  B_{q'}^{\alpha_2}$ and
$\overline{B} \in B_{p/r}^{\alpha_1 r}\,$,
and  $w,v$  weights such  that for all cubes $Q$,
\begin{equation}
  |Q|^{r/q-r/p} \, \|w^{r/p}\|_{L^A(Q)}\, \|v^{-r/p}\|_{L^B(Q)}\le c < \infty \,.
\end{equation}
Then,  if  the $\lambda_m$ as defined in Lemma 6.1 satisfy
\[\sum_{m=1}^{\infty} \lambda_m 2^{mn  /q} < \infty\,,\]
 we have
\begin{equation}
\Big(\int_{\R^n}   | Tf(x)   |^q\, w(x) \, dx\Big)^{1/q} \le
 c \,\Big(\int_{\R^n} |f(x)|^p\, v(x)\, dx\Big)^{1/p}
\end{equation}
for  those $f$ such that $\lim_{Q_0\to\R^n} m_{Tf}(t,Q_0)=0$.
\end{theorem}

\begin{proof}

We begin by considering the local version of (6.17).
Fix a cube $Q_0$ and note that by Theorem 6.5,
\begin{align*}
|Tf(x) - &m_{Tf}(t,Q_0)|
\\
&\le  8\,M^{\sharp}_{0,s,Q_0}(Tf)(x) +  c  \sum_{v,j} m^\sharp_{Tf}   (1-(1-t)/2^n,\widehat{Q^v_j}  )\indicator{Q^v_j}(x) \,,
\end{align*}
and therefore to estimate the $L^q_w(Q_0)$ norm of $Tf(x) - m_{Tf}(t,Q_0)$ it suffices
to estimate the norm of each summand separately.
Since by Theorem 4.1 or Theorem 4.3   we have
\[M^{\sharp}_{0,s,Q_0}  (Tf )(x)\le c\, M_r f(x)=c\, M   (|f|^r  )(x)^{1/r}\,,\]
the first term above can be estimated in norm by
\[\| M   (|f|^r )^{1/r} \|_{L^q_w}=\| M(|f|^r)\|_{L^{q/r}_w}^{1/r}\,.\]
Now, since   $\overline A\in B_{(q/r)'}$, by (6.13),
$ \|w^{r/q}\|_{L^{q/r}(Q)}\le c\,\|w^{r/q}\|_{L^A(Q)}$
for all cubes $Q$,
 and therefore (6.16) implies (6.14) with indices $p/r$ and $q/r$ there. Thus,
\[\| M(|f|^r)\|_{L^{q/r}_w}^{1/r}\le c\,\|\, |f|^r\|_{L^{p/r}_v}^{1/r}=c\, \| f\|_{L^{p}_v}\]
and
\[ \|M^{\sharp}_{0,s,Q_0}  (Tf )\|_{L^q_w}\le c\, \| f\|_{L^{p}_v}\,.
\]

Next, note that by a  geometric argument, if $Q$ is any of the cubes $Q^v_j$, there is a dimensional constant $c$
 such that
\begin{align}
\sum_{m=1}^{\infty} \lambda_m\Big( \frac{1}{|2^m \widehat{Q}|} \int_{2^m \widehat{Q}}
|f(y)|^r \, &dy\Big)^{1/r}
\notag
\\
&\le c \sum_{m=1}^{\infty}\lambda_m\,\Big(\frac{1}{|2^m Q|}\int_{2^m Q} |f(y)|^r\,dy\Big)^{1/r}.
\end{align}

To    estimate the norm of the sum by duality, let $h$ be such that ${\rm supp}(h) \subset Q_0$ and $\|h\|_{L^{q'}(Q_0)} = 1$,
and note that by (6.8) and (6.18),
\begin{align}\int_{Q_0}  \Big(  &\sum_{v,j}  m^\sharp_{Tf}  (1-(1-t)/2^n,\widehat{Q^v_j} ) \indicator{Q^v_j}(x) \Big )\, w (x)^{1/q} \, h(x)\,dx
\notag\\
&\le c\sum_m\lambda_m \sum_{v,j}\Big ( \frac{1}{|2^{m}Q_j^v|} \int_{2^m Q_j^v} |f(y) |^r \, dy \Big )^{1/r} \int_{Q_j^v} w (x)^{1/q}\,  h(x)\,dx\,.
\end{align}

We consider each term in the inner sum of (6.19) separately. First,
let $D$ be the Young function defined by $D(t)={\overline{B}}(t^r)$, and note
that  by H\"older's inequality for the conjugate Young functions $B,\overline B$
and  (6.12),
\begin{align*}
\Big ( \frac{1}{|2^{m}Q_j^v|} \int_{2^m Q_j^v} &|f(y) |^r \, dy \Big )^{1/r}
\\
&=
\Big ( \frac{1}{|2^{m}Q_j^v|} \int_{2^m Q_j^v} |f(y) |^r \,v(y)^{r/p} \,v(y)^{-r/p} \, dy \Big )^{1/r}\\
&\le 2\, \big(   \||f|^r v^{r/p} \|_{L^{\overline{B}}(2^mQ^v_j)} \, \| v^{-r/p} \|_{L^{B}(2^mQ^v_j)} \big)^{1/r} \,,
\\
&= 2\, \| f v^{1/p} \|_{L^{D}(2^mQ^v_j)} \,  \| v^{-r/p} \|_{L^{B}(2^mQ^v_j)}^{1/r}\,.
\end{align*}

Next, let $C$ be the Young function defined by $C(t)= A(t^r)$  and note
that   by H\"older's inequality for the conjugate Young functions $C,\overline C$
and  (6.12),
 \begin{align*}
 \int_{Q_j^v} w (x)^{1/q} \, h(x)\,dx &\le
 2^{mn} |Q_j^v|\, \frac{1}{|2^{m}Q_j^v|} \int_{2^m Q_j^v} w(x)^{1/q} \, h(x)\indicator{Q_j^v}(x)\,dx
\\
&\le 2\cdot 2^{mn}\,  \|w^{1/q}\|_{L^{C}(2^m Q^v_j)} \|h \indicator{Q^v_j}\|_{L^{\overline{C}}(2^m Q^v_j)}\,|Q_j^v|
\\
&\le 2\cdot 2^{mn}\, \|w^{r/q}\|_{L^{A}(2^m Q^v_j)}^{1/r} \|h \indicator{Q^v_j}\|_{L^{\overline{C}}(2^m Q^v_j)}\,|Q_j^v|\,.
\end{align*}

Moreover, since for each $\lambda>1$ and each cube $Q$ we have
\[  \|g \indicator{Q}  \|_{L^{\overline{C}}(\lambda  Q)} \le  \|g   \|_{L^{\overline{C} /\lambda^n}(Q)}\,,
\]
it follows that
\[ \int_{Q_j^v} w (x)^{1/q} \, h(x)\,dx \le 2\cdot
2^{mn} \, \|w^{r/q}\|_{L^{A}(2^m Q^v_j)}^{1/r}\, \|h \|_{L^{{\overline{C}/2^{mn}}}(  Q^v_j)}\, |Q_j^v|\,.\]

Therefore,
since by (6.16)   with $1/p-1/q=\alpha$,
\[ \|w^{r/q}\|_{L^{A}(2^m Q^v_j)}^{1/r}\,\| v^{-r/p} \|_{L^{B}(2^mQ^v_j)}^{1/r}\le c\,  |2^mQ^v_j|^{\alpha}\,,
\]
each term in the inner sum of (6.19) does not exceed
\[ c\,2^{mn} \, | 2^m Q^v_j|^{\alpha } \,\| f\,  v^{1/p} \|_{L^{D}(2^mQ^v_j)}
\,  \|h \|_{L^{{\overline{C}/2^{mn}}}(  Q^v_j)}\,|Q_j^v|\,,
\]
and consequently the sum itself is bounded  by
\begin{equation}
c\, \sum_{m=1}^\infty \lambda_m
 2^{mn} \, \sum_{v,j} | 2^m Q^v_j|^{\alpha}\, \|  f v^{1/p} \|_{L^{D}(2^m Q^v_j)}\,
 \|h    \|_{L^{\overline{C}/2^{mn}}(  Q^v_j)}\,  |Q_j^v|\,.
\end{equation}

 Let $F^v_j = Q^v_j \setminus \Omega^{v+1}$;  as in Theorem 3.2
it follows that the $F^v_j$ are pairwise disjoint and $|F^v_j| \ge  c |Q^v_j|$,
where $c$ depends on $s$ and $t$  but is independent of $v$ and $j$.
Then,  since $\alpha=\alpha_1+\alpha_2$, the innermost sum in (6.20) is bounded by
\[ J= c\, \sum_{v,j}   | 2^m Q^v_j|^{\alpha_1}\,
 \| f v^{1/p}  \|_{L^{D}(2^m  Q^v_j)}  \,  | 2^m Q^v_j|^{\alpha_2}\, \|h    \|_{L^{\overline{C}/2^{mn}}(  Q^v_j)} |F^v_j|\,,
\]
and, since
\[|2^m Q^v_j|^{\alpha_1}\, \|  f v^{1/p} \|_{L^{D}(2^m Q^v_j)}\le \inf_{x\in F_j^v}
M_{\alpha_1, D} (f v^{1/p})(x)\]
and similarly
\[ |2^m Q^v_j|^{\alpha_2}\,\|h    \|_{L^{\overline{C}/2^{mn}}(  Q^v_j)}\le
\inf_{x\in F_j^v} M_{\alpha_2,{\overline{C}/2^{mn}}}h(x)\,,
\]
we have that
\begin{align*}
   J &\le  c   \sum_{v,j} \int_{F^v_j}  M_{\alpha_1, D}  (f v^{1/p} )(x) \, M_{\alpha_2,\overline{C}/2^{mn}}\, h(x) \, dx
\\
&\le  c   \int_{Q_0}M_{\alpha_1,D}  (f v^{1/p} )(x) \, M_{\alpha_2,\overline{C}/2^{mn}}\,h(x) \, dx\,.
\end{align*}

Now pick $s_1,s_2$ such that
\begin{equation}
 1/p-\alpha_1 =1/s_1\,,\quad{\text{and}}\quad 1/q'-\alpha_2 =1/s_2\,.
\end{equation}
Since
\[ 1/s_1+1/s_2 =1/p-\alpha_1+ 1-1/q -\alpha_2 = 1/p-\alpha -1/q +1=1\,,
\]
$s_1,s_2$ are conjugate exponents, and, therefore, by H\"older's inequality,
\begin{align*}
\int_{Q_0}M_{\alpha_1,D}  (f v^{1/p} )(x)   &M_{\alpha_2,\overline{C}/2^{mn}}\, h(x) \, dx
\\
&\le
\| M_{\alpha_1,D}  (f v^{1/p} )\|_{L^{s_1}}  \| M_{\alpha_2,\overline{C}/2^{mn}}\,h\|_{L^{s_2}} .
\end{align*}

Now,  by (iii) and (iv) in
Proposition 6.1,  $D \in B_{p}^{\alpha_1} $ and  $\overline{C} \in B_{q'}^{\alpha_2}$, respectively,
and therefore by (6.11) and Theorem 3.3 in \cite{CUM},
\[ \int_{Q_0}M_{\alpha_1,D}  (f v^{1/p} )(x)   M_{\alpha_2,\overline{C}/2^{mn}}\,h(x) \, dx\le
 c\, \|f v^{1/p} \|_{L^p} \,  2^{-{mn/q'}}\, {\|h\|_{L^{q'}(Q_0)}}   \,,
 \]
and   the right-hand side of (6.20) is bounded by
\[ c \,\Big ( \sum_{m=1}^{\infty} \lambda_m 2^{mn(1 - 1/q')} \Big )   \|f  \|_{L^p_v} \le c\, \|f  \|_{L^p_v}\,.
\]
Hence,  combining the above estimates,
\[ \| Tf - m_{Tf}(t,Q_0) \|_{L^q_w(Q_0)}\le c \,\|f  \|_{L^p_v}\,.
\]

Finally,   by Fatou's lemma,
(6.17) follows for functions $f$ such that $m_{ Tf}(t,Q_0) \to 0$ as $Q_0 \to \R^n$.
\end{proof}

Observe that   $  B_{q'}\subset B_{(q/r)'}\cap  B_{q'}^{\alpha_2}$  and that, as the example
\[
\Phi(t) = \frac{t^{q' } } {\log(t)^{(1+\varepsilon)q' / s}} \,
\]
where $\varepsilon < q' /s - 1$ and $1 < q' \le  s < \infty$ is defined as $1/s = 1/q' - \alpha_2$, shows,
the inclusion is proper \cite{CUM}.
Now, in the case $p=q$, Theorem 6.6 builds on Theorem 1.3 in \cite{Lerner2010}, where
Lerner proves that for $n < p < \infty$, the full validity of a result anticipated by
 Cruz-Uribe and P\'erez \cite{CUP} holds
for singular integrals with $\omega(t)=t$; the sharpness of this result is discussed in \cite{CUMP, CUP}.
Theorem 6.6 in particular gives that if $r=1$ and $\omega(t)=t^\eta$ with $0<\eta<1$, then
the continuity holds for $ n/\eta<p<\infty$, and that in the case of kernels that satisfy a Dini condition,
whenever
\[\int_0^1 \omega_{r'}(c_nt) \frac1{t^{n/p}}\,\frac{dt}{t}<\infty\,,
\]
where $1\le r<\infty$.

\section{Morrey Spaces}

For a Young function $\Phi$ and  a continuous function $\phi(x,t)$ on $\R^n\times \R^+$
such that for each $x\in\R^n$,  $\phi(x,t) $ is increasing for $t$ in $[0,\infty)$ and $\phi(x,0)=0$, with $Q=Q(x_Q,l_Q)$, let
\[ \|f\|_{\mathcal M^{\Phi, \phi}_w}=\sup_{Q(x_Q,l_Q)\subset \R^n} \phi(x_Q, l_Q)\,\|f\|_{L^\Phi_w(Q)}\,.\]

Note that if $w =1$, $\Phi(t)=t^p$, and $\phi(x,t)=t^{n/p_0}$ for $1\le p\le p_0$, then $\mathcal M^{\Phi,\phi}$ $=\mathcal M^{p,p_0}$,
the familiar Morrey space.

As for the Campanato spaces $\mathcal L^{\Phi, \phi}_w$, consider the seminorms
\[ \|f\|_{\mathcal L^{\Phi, \phi}_w} =\sup_{Q(x_Q,l_Q)\subset \R^n} \phi(x_Q, l_Q)\,\|f-m_f(t,Q)\|_{L^\Phi_w(Q)}\,.\]

Although a priori the functions $\Phi$ and $\phi$ are unrelated, even in the simplest case
there are some limitations \cite{SawST}. In the unweighted case and when $\phi$ is independent of $x$,
in order that the characteristic function of the unit cube belongs to $\mathcal M^{\Phi,\phi}$  we assume that
\[ \sup_{t>1} \frac{\phi(t)}{\Phi^{-1}(t^n)}<\infty\,.\]

As pointed out in the introduction,  if $T$ is a Calder\'on-Zygmund singular integral operator that satisfies the
conditions of Theorem 4.1, for every $w\in A_\infty$ and Young function $\Phi$,
\[ \|Tf\|_{\mathcal L^{\Phi, \phi}_w}\le c\,\|M f\|_{\mathcal M^{\Phi, \phi}_w}\,.\]

The question is then  to remove the maximal function on the right-hand side of the above inequality.
Let $S$ be a  sublinear operator such that for a weight $w$,
\begin{equation} \int_{\R^n} \Phi(|Sf(y)|)\,w(y)\,dy\le c\int_{\R^n} \Phi(|f(y)|)\, w(y)\,dy\,,
\end{equation}
 and  for any cube $Q$, if $x\in Q$ and  ${\text{supp}}(f)\subset \R^n\setminus 2Q$, then
\begin{equation}
\ |Sf(x)|\le \int_{\R^n} \frac{|f(y)|}{|x-y|^n}\,dy\,.
\end{equation}
Such operators are considered, for instance, in \cite{GuliShuku}, and they include the maximal function as well as
a variety of singular integral operators..

Let then $\Phi$ be a Young function with $p=1/u_\phi$ and
$w\in A_p$. These weights satisfy condition $A_\Phi$,
namely, for every $\varepsilon>0$ and cube $Q$, with $\Phi(t)=\int_0^t a(s)\,ds$,
\begin{equation}
\Big(\frac1{|Q|}\int_Q \varepsilon w(y)\,dy\Big) a\Big(\frac1{|Q|}\int_Q a^{-1}\Big(\frac1{\varepsilon w(y)}\Big)\,dy\Big)\le c\,,
\end{equation}
which is equivalent to the integral inequality (7.1) for $S=M$, the Hardy-Littlewood maximal function \cite{KermanTorchinsky}.

Now, if $w\in A_\Phi$, by H\"older's inequality,
\begin{align} \frac1{|Q|}\,\int_{Q}|f(y)|\, dy
&= \varepsilon \, \frac{w(Q)}{|Q|}  \,\frac{1}{w(Q)}\,\int_{Q}|f(y)|\,\frac1{\varepsilon w(y)}\,w(y)\, dy\notag\\
&\le  2\, \varepsilon \, \frac{w(Q)}{|Q|} \,\| 1/ \varepsilon \,w\|_{L^{\overline{\Phi}}_w(Q)}\, \| f\|_{L^{\Phi}_w(Q)}\,.
\end{align}

We claim that for the choice $\varepsilon= \| 1/ w\|_{L^{\overline{\Phi}}_w(Q)}$,
\begin{equation}
 \varepsilon \, \frac{w(Q)}{|Q|} \le c\,,
\end{equation}
with a constant independent $c$ of $Q$.

Indeed, since $\overline \Phi(s)\sim s\, a^{-1}(s)$,  we have
\[1\sim \frac1{w(Q)}\int_Q\overline \Phi\Big(\frac1{\varepsilon w(y)}\Big)\,w(y)\,dy
\sim \frac1{w(Q)}\int_Q a^{-1}\Big(\frac1{\varepsilon w(y)}\Big)\frac1{\varepsilon}\,dy \,, \]
and therefore,
\[a\Big( \frac1{|Q|}\int_Q a^{-1}\Big(\frac1{\varepsilon w(y)}\Big)\,dy\Big)
\sim a\Big(\varepsilon\frac{w(Q)}{|Q|}\Big)\,,
\]
which, by  (7.3) gives
\begin{align*}
\Phi\Big(\varepsilon\frac{w(Q)}{|Q|}\Big)
&\sim a\Big(\varepsilon\frac{w(Q)}{|Q|}\Big) \varepsilon\frac{w(Q)}{|Q|}\\
&\sim \varepsilon\frac{w(Q)}{|Q|} a\Big( \frac1{|Q|}\int_Q a^{-1}\Big(\frac1{\varepsilon w(y)}\Big)\,dy\Big)\le c\,,
\end{align*}
and  (7.5) holds.

Thus, (7.4) gives
\begin{equation}
\frac1{|Q|}\,\int_{Q}|f(y)|\, dy\le c \,\| f\|_{L^\Phi_w(Q)}\,.
\end{equation}

We then have
\begin{theorem}
Let $S$ be a sublinear operator that satisfies {\rm (7.1)} and {\rm (7.2)}, $\Phi$ a Young function
so that $0<u_\Phi=1/p<1$, $w\in A_p$, and   $\phi(x,t), \psi(x,t)$ such that
for all $x\in\R^n$ and $\, l>0$,
\begin{equation}
{\psi(x, l)}\,\int_{ l}^\infty \frac1{\phi(x,t)}\, \frac{dt}{t}\le c\,.
\end{equation}
Then
\[ \|Sf\|_{\mathcal M^{\Phi, \psi}_w }\le c\,\|f\|_{\mathcal M^{\Phi, \phi}_w}\,.
\]
\end{theorem}

\begin{proof}  Fix   a cube $Q=Q(x_Q,l_Q)$  of $\R^n$, and for a function $f\in \mathcal M^{\Phi, \phi}_w $,
let $f=f_1+f_2$  where $f_1=f\chi_{2Q}$ and $f_2=f-f_1$. Then
\[\|Sf\|_{L^\Phi_w(Q)}\le \|Sf_1\|_{L^\Phi_w(Q)}+\|Sf_2\|_{L^\Phi_w(Q)}\,.\]

Now, by (7.1),
\begin{align*}
\int_{Q} \Phi(|Sf_1(y)|)\,w(y)\,dy &\le c\int_{\R^n} \Phi(|f_1(y)|)\,w(y)\,dy\\
& = c\int_{2Q} \Phi(|f(y)|)\,w(y)\,dy\,,
\end{align*}
and so
\[ \frac1{w(Q)}\,\int_{Q} \Phi(|Sf_1(y)|)\,w(y)\,dy
\le c \, \frac1{w(2\,Q)} \int_{2Q} \Phi(|f(y)|)\, w(y)\,dy
\]
which readily gives
\begin{equation}
\|Sf_1\|_{L^\Phi_w(Q)}\le c\, \|f\|_{L^\Phi_w(2Q)} \le c\,
\int_{2l_Q}^\infty   \| f\|_{L^\Phi_w(Q(x_Q,t))}  \, \frac{dt}{t}\,.
\end{equation}

We next deal with the term with $f_2$. Note that
 for $x\in Q $ and $y\notin 2Q$, $|x-y|\sim |x_Q-y|$, and therefore
\[ |Sf_2(x)|\le c \int_{\R^n \setminus 2Q} \frac{|f(y)|}{|x_Q-y|^n}\, dy \,.
 \]

Now, by Fubini's theorem,
\begin{align*}
\int_{\R^n \setminus 2Q} \frac{|f(y)|}{|x_Q-y|^n}\, dy &\le c
\int_{\R^n \setminus 2Q} {|f(y)|}\int_{|x_Q-y|}^\infty\,\frac1{t^n} \frac{dt}{t}\,dy
\\
&\le c \int_{2l_Q}^\infty \int_{Q(x_Q,t )\setminus Q(x_Q,2 l_Q)}|f(y)|\, dy\, \frac1{t^n} \frac{dt}{t}
\\
&\le c \int_{2l_Q}^\infty \frac1{|Q(x_Q,t)|}\,\int_{Q(x_Q,t )}|f(y)|\, dy\,  \frac{dt}{t}\,,
\end{align*}
and consequently  by (7.6),
\[
 |Sf_2(x)| \le c   \,\int_{2l_Q}^\infty   \| f\|_{L^\Phi_w(Q(x_Q,t))}  \, \frac{dt}{t}\,.
 \]

Moreover, since for every $Q,\Phi,w$,
\[ \| g\|_{L^\Phi_w(Q)} \le c\,\| g\|_{L^\infty(Q)}\,,\]
it follows that
\[ \|Sf_2\|_{L^\Phi_w(Q)} \le c\, \int_{2l_Q}^\infty   \| f\|_{L^\Phi_w(Q(x_Q,t))}  \, \frac{dt}{t}\,, \]
which combined with (7.8)  gives
\[ \|Sf\|_{L^\Phi_w(Q)} \le c\, \int_{2l_Q}^\infty   \| f\|_{L^\Phi_w(Q(x_Q,t))}  \, \frac{dt}{t}\,.
\]
Therefore by  (7.7),
\begin{align}
 \|Sf\|_{L^\Phi_w(Q)} &\le c  \,\int_{2l_Q}^\infty \phi(x_Q,t)\, \| f\|_{L^\Phi_w(Q(x_Q,t))} \frac1{\phi(x_Q,t)}\, \frac{dt}{t}\notag\\
 &\le c\,\Big( \int_{2l_Q}^\infty \frac1{\phi(x_Q,t)}\, \frac{dt}{t}\Big)\, \|f\|_{\mathcal M^{\Phi, \phi}_w}\notag\\
 &\le c\,   \|f\|_{\mathcal M^{\Phi, \phi}_w}\, \frac1{\psi(x_Q,2l_{Q})}\,.
\end{align}

Now, since  $\psi$ is increasing, from (7.9) it follows that \[\psi(x_Q,l_{Q})\, \|Sf \|_{L^\Phi_w(Q)} \le c\, \|f\|_{\mathcal M^{\Phi, \phi}_w}\]
and so, taking the supremum over $Q$,
\[ \|Sf\|_{\mathcal M^{\Phi, \psi}_w} \le  c\,  \|f\|_{\mathcal M^{\Phi, \phi}_w}\,.\]
The proof is thus complete.
\end{proof}


\begin{thebibliography}{HD}


\baselineskip=17pt

\bibitem{Adams}
D.~R.~Adams and J.~Xiao,
\emph{Morrey spaces in harmonic analysis},
 Ark. Mat.  \textbf{50}  (2012),  no. 2, 201--230.



\bibitem{AlvarezPerez}
J.~Alvarez and C.~P\'{e}rez, \emph{Estimates with $A_{\infty}$ weights for
  various singular integral operators}, Boll. Un. Mat. Ital. A (7) \textbf{8}
  (1994), no.~1, 123--133.

\bibitem{Anderson}
T.~C.~Anderson and  A.~Vagharshakyan, \emph{A simple proof of the sharp weighted estimate for Calder\'on-Zygmund
operators on homogenous spaces}, J. Geom. Anal.
DOI 10.1007/s12220-012-9372-7.



\bibitem{GoodLambda}
D.~L. Burkholder and R.~F. Gundy, \emph{Extrapolation and interpolation of
  quasilinear operators on martingales}, Acta Math. \textbf{124} (1970),
  249--304.

\bibitem{Calderon1972}
A.~P. Calder\'{o}n, \emph{Estimates for singular integral operators in terms of
  maximal functions}, Studia Math. \textbf{44} (1972), 563--582.

\bibitem{Compo} M. Carozza and A. Passarelli Di Napoli,
\emph{Composition of maximal operators},
Publ. Mat.  \textbf{40} (1996), no. 2, 397-–409.


\bibitem{RRC}   R.~R.~Coifman, \emph{Distribution function inequalities for singular integrals},
Proc. Nat. Acad. Sci. U.S.A. \textbf{69} (1972), 2838--2839.

\bibitem{CoifmanFefferman}
R.~R. Coifman and C.~Fefferman, \emph{Weighted norm inequalities for maximal
  functions and singular integrals}, Studia Math. \textbf{51} (1974), 241--250.

\bibitem{CordobaFefferman}
A.~C\'{o}rdoba and C.~Fefferman, \emph{A weighted norm inequality for singular
  integrals}, Studia Math. \textbf{57} (1976), no. 1,  97--101.

\bibitem{MC}
M.~Cotlar, \emph{Some generalizations of the Hardy-Littlewood maximal theorem},
Rev. Mat. Cuyana \textbf{1} (1955), 85--104 (1956).



\bibitem{CUMP}
D.~Cruz-Uribe, SFO, J.~M.~Martell, and C.~P\'erez,
\emph{Sharp two-weight inequalities for singular integrals, with
applications to the Hilbert transform and the Sarason conjecture}, Adv. Math. \textbf{216} (2007), no. 2,  647-–676.

\bibitem{CUMP1}
\bysame,
\emph{Sharp weighted estimates for classical operators}, Adv. Math.
\textbf{229} (2012), no. 1,   408-–441.

\bibitem{CUM}
D.~Cruz-Uribe, SFO and K.~Moen, \emph{A fractional Muckenhoupt-Wheeden theorem and its consequences},
Integr. Equ. Oper. Theory  \textbf{76}  (2013),  no. 3, 421-–446.



\bibitem{CUP}
D.~Cruz-Uribe, SFO and C.~P\'erez,
\emph{On the two-weight problem for singular integral operators}, Ann. Scu.
Norm. Super. Pisa Cl. Sci. (5) \textbf{1} (2002) no. 2,  821--849.




\bibitem{FeffermanStein}
C.~Fefferman and E.~Stein, \emph{Some maximal inequalities}, Amer. J. Math.
  \textbf{93} (1971), no.~1, 107--115.

\bibitem{FSActa}
\bysame, \emph{ $H^p$ spaces of several variables}, Acta Math. \textbf{129} (1972), no. 3-4,
137--193.


\bibitem{Fujii1989}
N.~Fujii, \emph{A proof of the {F}efferman-{S}tein-{S}tr\"{o}mberg inequality
  for the sharp maximal function}, Proc. Amer. Math. Soc. \textbf{106} (1989),
  no.~2, 371--377.

\bibitem{Fujii1991}
\bysame,
\emph{A condition for the two-weight norm inequality for singular
  integral operators}, Studia Math. \textbf{98} (1991), no.~3, 175--190.


\bibitem{GCM}
J.~Garc\'ia-Cuerva and  J.~M. Martell,
\emph{Two-weight norm inequalities for maximal operators and fractional integrals on
non-homogeneous Spaces}, Indiana Univ. Math. J.  \textbf{50}  (2001),  no. 3, 1241-–1280.

\bibitem{GJ}
J.~B.~Garnett and P.~W.~Jones,
\emph{BMO from dyadic BMO}, Pacific J. Math. \textbf{99} (1982),
no. 2, 351-–371.

\bibitem{GuliNaz2006} V.~S.~Guliyev and Sh.~A.~ Nazirova, \emph{Two-weighted inequalities
for some sublinear operators in Lebesgue spaces}, Khazar Journal of Mathematics \textbf{2}
 (2006), no. 1, 3--22.






\bibitem{GuliShuku}
 V.~S.~Guliyev,  S.~S.~Aliyev, T.~Karaman,
and P.~S.~Shukurov,
 \emph{Boundedness of sublinear operators and commutators on generalized
Morrey spaces}, Integr. Equ. Oper. Theory \textbf{71} (2011), no. 3,  327--355.


\bibitem{HMW}
R.~Hunt, B.~Muckenhoupt, and R.~L.~Wheeden, \emph{Weighted norm inequalities for
  the conjugate function and {H}ilbert transform}, Trans. Amer. Math. Soc.
  \textbf{176} (1973), 227--251.

\bibitem{JawerthTorchinsky}
B.~Jawerth and A.~Torchinsky, \emph{Local sharp maximal functions}, J. Approx.
  Theory \textbf{43} (1985), no.~3, 231--270.

\bibitem{KermanTorchinsky}
R.~A. Kerman and A.~Torchinsky, \emph{Integral inequalities with weights for
  the {H}ardy maximal function}, Studia Math. \textbf{71} (1981/82),  no. 3, 277--284.

\bibitem{KurtzW}
D.~S.~ Kurtz and R.~L.~Wheeden, \emph{Results for weighted norm inequalities for multipliers},
Trans. Amer.  Math. Soc.  \textbf{255} (1979), 343--362.


\bibitem{L} A.~K.~Lerner, \emph{On the John-Str\"omberg characterization of BMO for nondoubling
measures}, Real. Anal. Exchange \textbf{28} (2002/2003), no. 2, 649--660.


\bibitem{Lerner2010}
\bysame,
 \emph{A pointwise estimate for the local sharp maximal function with
  applications to singular integrals}, Bull. Lond. Math. Soc. \textbf{42}
  (2010), no.~5, 843--856.

\bibitem{LernerSummary}
\bysame,
\emph{A ``local mean oscillation'' decomposition and some of its
  applications}, {F}unction {S}paces, {A}pproximation, {I}nequalities and
  {L}ineability, Lectures of the Spring School in Analysis, Matfyzpres, Prague (2011), pp.~71--106.


\bibitem{La}
\bysame,
\emph{A simple proof of the $A_2$ conjecture}, Int. Math. Res. Not. (2012) DOI 10.1093/imrn/rns145.

\bibitem{Lorente}
M.~Lorente, M.~S. Riveros, and A.~de~la Torre, \emph{Weighted estimates for
  singular integral operators satisfying {H}\"{o}rmander's conditions of {Y}oung
  type}, J. Fourier Anal. Appl. \textbf{11} (2005), no.~5, 497--509.


\bibitem{Muck}
  B.~Muckenhoupt, \emph{ Norm inequalities relating the {H}ilbert transform to the {H}ardy-{L}ittlewood maximal function},
  Functional analysis and approximation (Oberwolfach, 1980),   pp. 219--231.

\bibitem{MuckWhee}
B.~Muckenhoupt and  R.~L.~Wheeden,
\emph{Two weight function norm inequalities for the {H}ardy-{L}ittlewood maximal function and
the {H}ilbert transform}, Studia Math. \textbf{55} (1976), 279--294.





\bibitem{Perez1994}
C.~P\'{e}rez,
 \emph{Weighted norm inequalities for singular integral operators}, J.
  London Math. Soc. (2) \textbf{49} (1994), no.~2, 296--308.

\bibitem{PerezIUMJ}
\bysame,
\emph{Two weighted inequalities for potential and fractional type maximal operators},
Indiana Univ. Math. J. \textbf{43}, (1994), 1--28.


\bibitem{Perez1990}
\bysame,
 \emph{On sufficient conditions for the boundedness of the
  {H}ardy-{L}ittlewood maximal operator between weighted $L^p$-spaces with
  different weights}, Proc. London Math. Soc. (3) \textbf{71}  (1995), no.~1, 135--157.





\bibitem{MedContOsc}
J.~Poelhuis and A.~Torchinsky, \emph{Medians, continuity, and vanishing
  oscillation}, Studia Math. \textbf{213}   (2012),  no. 3, 227--242.



\bibitem{ReSc}
M.~C.~Reguera  and J.~Scurry,
\emph{On joint estimates for maximal functions and singular integrals on
weighted spaces}, Proc. Amer. Math. Soc. \textbf{141} (2012), no.~5, 1705--1717.



\bibitem{RiverosUrciuolo}
M.~S. Riveros and M.~Urciuolo, \emph{Weighted inequalities for integral
  operators with some homogeneous kernels}, Czechoslovak Math. J. (130) \textbf{55} (2005), no.~2, 423--432.

\bibitem{RiverUrci}
\bysame,
\emph{Weighted inequalities for fractional type operators with some homogenous
kernels}, Acta Math. Sin. (Engl. Ser.),  \textbf{29}, no. 3, 449--460.




\bibitem{RTriebel}
M.~Rosenthal and  H.~Triebel, \emph{Calder\'on-{Z}ygmund operators in {M}orrey spaces},
Rev. Mat. Complut. (2013), DOI 10.1007/s13163-013-0125-3.



\bibitem{SawST}
Y.~Sawano, S.~Sugano, and H.~Tanaka,
\emph{ Generalized fractional integrals operators
and fractional maximal operators in the framework of
{M}orrey spaces}, Trans. Amer. Math. Soc. \textbf{363} (2011), no.~12, 6481--6503.


\bibitem{Sawyer1982}
E.~Sawyer, \emph{A characterization of a two-weight norm inequality for maximal
  operators}, Studia Math. \textbf{75} (1982), no.~1, 1--11.

\bibitem{SawProc}
\bysame,
\emph{Two weight norm inequalities for certain maximal and integral operators},
 Harmonic analysis (Minneapolis, Minn., 1981),  pp.~102--127, Lecture Notes in Math., 908, Springer-Berlin/New York, 1982.

\bibitem{Saw2}
\bysame,  \emph{Norm inequalities relating singular integrals and the maximal function},
 Studia Math.  \textbf{75}  (1983),  no. 3, 253--263.

\bibitem{ShiTor}
X.~L.~Shi and A. Torchinsky, \emph{Local sharp maximal functions in spaces of homogeneous type},
Sci. Sinica Ser. A  \textbf{30}  (1987),  no. 5, 473--480.


\bibitem {SP} S. Spanne, \emph{Some function spaces defined using the mean oscillation over cubes},
 Ann. Scuola Norm. Sup. Pisa (3) \textbf{19} (1965), 593-–608.


\bibitem{Stromberg}
J.-O. Str\"{o}mberg, \emph{Bounded mean oscillation with {O}rlicz norms and
  duality of {H}ardy spaces}, Indiana Univ. Math. J. \textbf{28} (1979), no.~3,
  511--544.



\bibitem{TorchinskyStromberg}
J.-O. Str\"{o}mberg and A.~Torchinsky, \emph{{W}eighted {H}ardy spaces},
  Lecture Notes in Mathematics, 1381, Springer-Verlag, Berlin, 1989.



\bibitem{TorInt} A. Torchinsky, \emph{Interpolation of operations and {O}rlicz classes}, Studia Math.
\textbf{59}  (1976/77),  177-–207.

\bibitem {TOR}
\bysame,
Real-variable methods in harmonic analysis, Pure and Applied Mathematics,
123, Academic Press, Inc., Orlando, FL, 1986. (Reprinted by Dover in 2004).

\bibitem{Yab}
K. Yabuta, \emph{Sharp maximal function and $C_p$ condition},
Arch. Math. (Basel) \textbf{55} (1990), no.~2, 151--155.



\end{thebibliography}
\end{document}